\documentclass[11pt]{article}
 \usepackage{amsmath,amsthm,amsfonts,amssymb}
\usepackage{verbatim}
\usepackage{latexsym}

\usepackage{graphicx}
\usepackage{caption}
\usepackage{subcaption}

\newtheorem{theorem}{Theorem}[section]
\newtheorem{lemma}[theorem]{Lemma}
\newtheorem{proposition}[theorem]{Proposition}
\newtheorem{definition}[theorem]{Definition}
\newtheorem{corollary}[theorem]{Corollary}
\newtheorem{conjecture}[theorem]{Conjecture}

\newtheorem{example}[theorem]{Example}

\newtheorem{remark}[theorem]{Remark}
\newenvironment{prf} {{\bf Proof.}}{\hfill $\Box$}

\newcommand\FF{{\mathbb{F}}}

\def\e{\epsilon}
\def\<{\langle}
\def\>{\rangle}

\def\bt{\Bbb T}
\def\bc{\Bbb C}

\def\ga{\gamma}
\def\vp{\varphi}
\def\ov{\overline}
\def\la{\lambda}
\def\br{\Bbb R}

\def\bn{\Bbb N}
\def\bz{\Bbb Z}
\def\bd{\Bbb D}
\def\minus{
  \setbox0=\hbox{-}
  \vcenter{
    \hrule width\wd0 height \the\fontdimen8\textfont3 }}
\def\one{(\mathbf{1})}

\def\al{\alpha}
\def\be{\beta}
\def\om{\omega}
\def\ov{\overline}
\def\re {{\hbox{Re}}}

\def\bi {{\hbox{Bi}}}
\newtheorem{observ}[theorem]{Observation}
\def\fix {{\hbox{Fix}}}
\def\diag {{\hbox{diag}}}
\def\si {{\hbox{sign}}}
\def\one{(\mathbf{1})}
\def\supp {{\hbox{supp}}}

\begin{document}
\title{On  biunimodular vectors for unitary matrices}
\author{
Hartmut F\"uhr, Ziemowit Rzeszotnik}
\maketitle

\abstract{A biunimodular vector of a unitary matrix $A \in U(n)$ is a vector $v \in \mathbb{T}^n\subset\bc^n$ such that $Av \in \mathbb{T}^n$ as well. Over the last 30 years, the sets of biunimodular vectors for Fourier matrices have been the object of extensive research in various areas of mathematics and applied sciences. 
Here, we broaden this basic harmonic analysis perspective and extend the search for biunimodular vectors to arbitrary unitary matrices. 
This search can be motivated in various ways. The main motivation is provided by the fact, that the existence of  biunimodular vectors for an arbitrary unitary matrix allows for a natural understanding of the structure of all unitary matrices.
}

\section* {Note} As we have learned from the referees, the existence of  biunimodular vectors for an arbitrary unitary matrix has been
recently proved in this journal  by Idel and Wolf (see \cite{IW}) basing on the results of \cite{BEP} in symplectic geometry due to Biran, Entov and Polterovich. In this paper we provide an independent account of this existence phenomenon and push a bit further the theory of  biunimodular vectors.
\section* {Outline}

In Section 1 we describe the background for this research. The next two sections discuss the structure of all unitary matrices via decompositions based on biunimodular vectors. In Section 4 we state the problem of finding biunimodular vectors in several equivalent terms. One of them is a maximal torus decomposition of $SU(n)$, that can be viewed as a close relative of the spectral theorem. In the following sections we justify that all matrices in $U(2)$ and $U(3)$ possess biunimodular vectors and provide two closed formulas for $U(3)$. One of them being an analogue of the standard Euler-Tait-Bryan decomposition of $O(3)$.  In Section 8 we provide a simplified novel method to synthesise all matrices in $U(2^n)$, that enjoys an elementary proof for all $n\in\bn$. Section 9 is devoted to the study of biunimodular vectors in the general case from the Lie groups perspective. Using basic tools we show that  in arbitrary dimension $n$, the set of matrices possessing biunimodular vectors contains a nonempty open 
subset of $U(n)$. We also obtain that, up to normalization, the set of biunimodular vectors is finite for almost every unitary matrix. Finally, we employ a simple alternating projection algorithm that for a numerically given unitary matrix produces vectors, that are biunimodular within a 
set precision. We test this algorithm to provide numerical evidence indicating its effectiveness in finding such near biunimodular vectors for random unitary matrices with dimensions up to 100. We also apply the algorithm to Fourier matrices, and compare the search results to the known solutions. 

\section{Background}\label{introduction}

The paper explores the notion of  a {\bf biunimodular vector} that traces back to Gauss. The classic example of such vectors are {\it Gauss sequences} given for $k=0,1,\dots, n-1$ as 
$$
u_k=ce^{\frac{2\pi i}n (\la k^2+\mu k)}  \hskip2cm \quad n \text{ odd},
$$ 
$$
u_k=ce^{\frac{2\pi i}n(\frac{\la}2 k^2+\mu k)} \hskip2cm  \quad n \text{ even},
$$ 
where $\la,\mu\in\{0,1,\dots, n-1\}$ with $\la$ relatively prime to $n\in\bn$ and $c\in\bt=\{z\in\bc:|z|=1\}$. These vectors are {\it unimodular}, that is, $|u_k|=1$ for all  $k=0,1,\dots, n-1$ and have an interesting property. Their Fourier transform $F_nu$, defined by aplying to $u$ the unitary $n\times n$ matrix $F_n$ with entries $(F_n)_{j,k}=\frac1{\sqrt n}e^{-2\pi i\frac{ jk}n}$, is unimodular as well.\footnote{ As a matter of fact, the Fourier transform of a Gauss sequence yields another Gauss sequence.}

Following Haagerup (\cite{Haa}) and Saffari (\cite{Sa}) we attribute the starting point of the general theory of biunimodular vectors to Per Enflo. In 1983 he has asked whether for prime $n$  the only unimodular vectors of length $n$ whose Fourier transform is also unimodular are  Gauss sequences. While the answer is affirmative when $n\le 5$, a computer search done by Bj\"orck for $n=7$ provided a surprising counterexample
$$
(1,1,1,e^{i\theta},1,e^{i\theta},e^{i\theta}), 
$$ 
with $\theta=\arccos (-\frac34)$ yielding $e^{i\theta}=-\frac{3}4+i\frac{\sqrt7}4$. These findings, published in 1985 (\cite{Bj0}),  have been generalized in \cite{Bj} and later the term ``bi-unimodular sequence'' has been coined by Bj\"orck and Saffari to describe a unimodular finite vector, whose Fourier transform is unimodular as well (see \cite{BS}).  

The Bj\"orck sequences provided in \cite{Bj} are biunimodular vectors of a prime length $p$. For $p\equiv -1 \pmod 4$ their coefficients are either $1$ or $e^{i\theta}$ with $\theta=\arccos\frac{1-p}{1+p}$.  For $p\equiv 1 \pmod 4$ their coefficients are either $1$, $e^{i\eta}$ or $e^{-i\eta}$ with $\eta=\arccos\frac{1}{\sqrt p +1}$. The importance of Bj\"orck sequences has been underlined in \cite{BBW}, where it was shown that for these sequences (contrary to Gauss sequences) the discrete narrow band
ambiguity function has an optimal bound. 
However, Gauss and Bj\"orck sequences provide only a glimpse at the structure of biunimodular vectors for $F_n$.
Even in the case $p=7$ there is another such vector found by Bj\"orck and  Fr\"oberg 
that is related neither to a Gauss nor to a Bj\"orck sequence (see \cite{BF} or \cite{Haa}).

Moreover, in \cite{BS} Bj\"orck and Saffari proved that if $n$ is divisible by a square, then there are infinitely many (continuum) biunimodular vectors with the leading entry 1, and they conjectured that for all square-free  $n$ the number of such biunimodular vectors is finite.
\begin{conjecture}\label{bs}{\rm(Bj\"orck, Saffari)} If $n\in\bn$ is not divisible by a square, then the set of biunimodular vectors (with the leading entry 1) for the Fourier matrix $F_n$ is finite.
\end{conjecture}

An easy observation is that a unimodular vector $u=(u_0,u_1,\dots,u_{n-1})$ is biunimodular for $F_n$ if and only if its cyclic translations are orthogonal. This orthogonality relation is called {\em zero autocorrelation} and means that
\begin{equation}\label{zac}
\sum_{k=0}^{n-1}u_k\ov{u_{k+i}}=0 \quad\text{ for } i=1,2,\dots,n-1,
\end{equation}
where the  index $k+i$ is taken $(\bmod n)$.

The idea of Bj\"orck to search for biunimodular vectors $u$ for $F_n$ relies on setting $x_k=\frac{u_{k+1}}{u_k}$, where $k=0,1,\dots,n-1$ with $u_n=u_0$ and transforming (\ref{zac}) to the following set of equations:
\begin{equation}\label{cnr}
\begin{split}
x_0+x_1+\dots+x_{n-1}=0\\
x_0x_1+x_1x_2+\dots+x_{n-1}x_0=0\\
\vdots\\
x_0x_1\dots x_{n-2}+\dots+x_{n-1}x_0\dots x_{n-3}=0\\
x_0x_1\cdot\dots\cdot x_{n-1}=1
\end{split}
\end{equation}
The unimodular solutions of the above set of equations are in one-to-one correspondence to biunimodular vectors for $F_n$  (with the leading entry 1). The general complex solutions of (\ref{cnr}) are called {\em cyclic $n$-roots} and became a research area of independent interest. In the mid-1980s Arnborg has asked Davenport to establish whether the set of cyclic 6-roots is finite, leading to the popularisation of the cyclic $n$-roots problem in \cite{Dav}. Backelin proved in \cite{Ba} that for $n$ divisible by a square the number of cyclic $n$-roots is infinite. For $2\le n\le 8$ all cyclic $n$-roots were found by Bj\"orck and Fr\"oberg (see \cite{BF, BF2, BaF}) by using Gr\"obner basis techniques to solve the system (\ref{cnr}). Faug\`ere has found cyclic 9 and 10 roots by improving the search for the Gr\"obner basis  (see \cite{Fa1} and \cite{Fa2}). These results have been confirmed and extended by polyhedral homotopy continuation methods applied for solving (\ref{cnr}) as a benchmark 
problem. The 
numbers of cyclic $n$-roots for a square-free $n$ between 2 and 14 is listed below.\footnote{For $2\le n\le 11$ the amount of roots has been verified by various methods, for $n=13$ and $14$ the number of cyclic $n$-roots is claimed only by Li and Tsai (see \cite{LT}). Moreover, they count only isolated roots, so the number of cyclic 14-roots still has to be verified as finite or not. } 

\begin{center}
\begin{tabular}{| c || c | c | c | c | c | c | c | c | c |c| }
    \hline
    $n$ & 2 & 3 &5&6&7&10&11&13&14&15\\ \hline
    cyclic $n$-roots & 2 &6& 70&156&924&34,940& 184,756& 2,704,156&  8,464,307&?\\ \hline
biuni vectors & 2 &6& 20&48&532&?& ?& 53,222&  ?&?\\ \hline
  \end{tabular}
\end{center}

The table also contains information about the number of unimodular cyclic $n$-roots that yields the amount of biunimodular vectors for $F_n$. The numbers up to $n=7$ were given by Bj\"orck, Fr\"oberg and Haagerup, who inspected the obtained cyclic $n$-roots to check for unimodular solutions, establishing the real benchmark for solving the cyclic $n$-root problem (see \cite{BF} and \cite{Haa}). The number of biunimodular vectors for $n=13$ has been obtained by Gabidulin and Shorin in \cite{GS}. 

Basing on the number of cyclic $p$-roots for prime $p$ up to 7 Fr\"oberg concluded that the amount of cyclic $p$-roots should be ${2p-2\choose p-1}$. This conjecture has been confirmed by Haagerup in \cite{Ha} under the condition that the roots are counted with multiplicities.

\begin{theorem}\label{hag}{\rm(Haagerup)} For prime $p$ the number of cyclic $p$-roots counted with multiplicities is  ${2p-2\choose p-1}$.
\end{theorem}

The paper of Haagerup contains an elegant direct argument showing that the number of cyclic $p$-roots is finite for prime $p$, confirming Conjecture~\ref{bs} in the prime case. His more advanced reasoning allows to count the cyclic $p$-roots with multiplicities, imposing a question whether for all primes cyclic $p$-roots have multiplicity 1 and providing an additional motivation to verify the number of distinct cyclic 13-roots. 

A slightly different approach towards finding biunimodular vectors for $F_n$ has been taken by Gabidulin and Shorin. In \cite{GS} they multiplied the system (\ref{zac}) by the product $\prod_{k=0}^{n-1}u_k$ to obtain a system of equations
that has been treated with Gr\"obner basis techniques as well, providing the mentioned number of biunimodular vectors for $F_{13}$.

The theory of biunimodular vectors may be hard to follow due to their popularity stemming from applications in several fields of signal processing. This popularity can be measured by the amount of various names given to the same objects by many researchers working within the area. Gauss sequences are often called Zadoff, Chu or Wiener sequences. Unimodular vectors are described as constant amplitude, phase-shift keyed or polyphase sequences. Vectors whose Fourier transform is unimodular are named as perfect, having optimum correlation or zero autocorrelation. Therefore,  biunimodular vectors  are studied as CAZAC (constant amplitude zero  autocorrelation) sequences, PSK (phase-shift keyed) perfect sequences, polyphase sequences with optimum correlation and so on (see \cite{ben} by Benedetto {\it et al.} for a better account of this nomenclature and a list of relevant papers  that can be extended by adding an early reference \cite{Al}). 

Since full understanding of biunimodular vectors would allow to resolve the circulant Hadamard matrix conjecture, and is just a tip of the iceberg being a classification of all Hadamard matrices, the complexity of the topic is apparent.

It is important to notice, that the research on biunimodular vectors is not confined to the Fourier transform $F_n$ on the cyclic group $\bz_n$. Indeed, real biunimodular vectors for the Fourier transform on the group $(\bz_2)^n$ are investigated under the name of {\it bent functions}. Here $n$ must be even and the number of bent functions is known for $n\le 8$. The latest count is due to Langevin and Leander who proved the number to be 99270589265934370305785861242880 in the case $n=8$ (see \cite{LL}).

In \cite{GiRz} it has been proved that biunimodular vectors exist for the Fourier transform on any finite abelian group. This raises the question of existence of such vectors for an arbitrary unitary matrix and motivates the following
\begin{definition} Let $A$ be an $n\times n$ unitary matrix. A vector $v\in\bc^n$  is called a {\bf biunimodular vector} for $A$ if both $v$ and $Av$ are unimodular (i.e. have all entries of modulus one).
\end{definition}

One can argue that the above notion has surfaced already in works on the problem of mutually unbiased bases (MUB). Indeed, the problem can be brought down to finding all biunimodular vectors for a given complex $n\times n$ Hadamard matrix and checking for a subcollection of these vectors that forms an orthonormal basis of $\bc^n$. This approach has been used in \cite{Gr} by Grassl for $F_6$ and  by Bengtsson {\it et al.} for some other  $6\times 6$ Hadamard matrices (see \cite{Be}). In \cite{BW1} Brierley and Weigert conducted an extensive search on all known Hadamard matrices in dimension 6 providing an impressive study of their biunimodular vectors leading to an approximate solution of the MUB problem in this dimension (for lower dimensions accessed via biunimodular vectors see  \cite{BW2}).

In the following sections we explain that the existence of biunimodular vectors for an arbitrary unitary matrix allows to understand the structure of unitary matrices in fairly simple terms.

\section {Synthesis of $U(n)$}\label{s3}

Unitary $n\times n$ matrices form the unitary group $U(n)$ and are a basic notion in mathematics, physics, chemistry and engeneering. These matrices can be obtained in various standard ways: Gram-Schmidt process, Householder reflections, Hurwitz parametrization, Cayley transform and via Hermitian matrices using the matrix exponential. In 1952 Murnaghan described a convenient parametrization of the unitary group (see \cite{Mu}).\footnote{Another such parametrization has been provided in 1982 by Di\c{t}\u{a} (see \cite{Dit1}).} This has been followed by the work of Reck {\it et al.} who in 1994 proved that any unitary matrix can be obtained as a sequence of consecutive beam splitter transformations, that can be conducted experimentally in the laboratory using optical devices (\cite{Reck}).  In the current century the research on finding the ways to generate unitary matrices has intensified. Nemoto described the special unitary group $SU(n)$ in \cite{Ne} basing on an earlier result \cite{Ro} of Rowe  {\it et al.
} Tilma  and Sudarshan 
provided an Euler angle 
parametrization for $SU(n)$ and $U(n)$ (see \cite{TS1, TS2}). Meanwhile  Di\c{t}\u{a}  in \cite{Dit2} has shown another description of $U(n)$. A recursive parametrization of unitary matrices has been given by Jarlskog in \cite{Jar} and a composite parameterization was done in \cite{SHH} (see \cite{Cab, Dit2}  and \cite{SHH} for more details on these developments).

The search for biunimodular vectors for an arbitrary unitary matrix can be motivated by the following recipe to generate unitary matrices. Consider $n\in\bn$ and a set of $n^2$ complex parameters $\{a_{jk}\in\bt:j,k=1,2,\dots, n\}$, where $\bt=\{z\in\bc:|z|=1\}$. In order to synthesize a unitary matrix $A_n\in U(n)$
from these parameters, it is convenient to organize them as in the table below

\[
\begin{array}{*{5}{c}}
\cline{1-3}\cline{5-5}
\multicolumn{1}{|c|}{a_{11}} &\multicolumn{1}{|c|}{ a_{12}} & \multicolumn{1}{|c|}{a_{13}}  &\dots& \multicolumn{1}{|c|}{a_{1n}}\\
\cline{1-2}
\multicolumn{1}{|c}{a_{21}}  & \multicolumn{1}{c|}{a_{22}} & \multicolumn{1}{|c|}{a_{23}} & \dots&\multicolumn{1}{|c|}{a_{2n}}\\
\cline{1-3}
\multicolumn{1}{|c}{a_{31}} & a_{32} & \multicolumn{1}{c|}{a_{33}} &\dots& \multicolumn{1}{|c|}{a_{3n}} \\
\cline{1-3}
\vdots &\vdots&\vdots&& \multicolumn{1}{|c|}{\vdots}\\
\cline{1-5}
\multicolumn{1}{|c}{a_{n1}} & a_{n2} &a_{n3}&\dots& \multicolumn{1}{c|}{a_{nn}}
  \\
\cline{1-5}

\end{array}
\]
and run the following recursive procedure using the Fourier transform  $(F_l)_{j,k}=\frac1{\sqrt l}e^{-2\pi i\frac{ jk}l}$ and its conjugate transpose
 $(F_l^*)_{j,k}=\frac1{\sqrt l}e^{2\pi i\frac{ jk}l}$ . 

The initial unitary matrix $A_1$ is given as $[a_{11}]$. And, in general, for $2\le l\le n$

\begin{equation}\label{syn}
 A_l = \begin{bmatrix}
a_{l1} & 0&\dots&0 \\
0 & a_{l2}&\dots&0\\
\vdots&\vdots&\ddots&\vdots\\
0&0&\dots&a_{ll}
\end{bmatrix}F_l
\left[ \: 
\begin{array}{*{4}{c}}
1&0&\dots&0\\
\cline{2-4}
0& \multicolumn{1}{|c}{} &  & \multicolumn{1}{c|}{} 
\\
\vdots& \multicolumn{1}{|c}{} & A_{l\minus 1} & \multicolumn{1}{c|}{} 
 \\
0&\multicolumn{1}{|c}{} &  & \multicolumn{1}{c|}{}
  \\
\cline{2-4}

\end{array}
\,\,\right] F_l^*
\begin{bmatrix}
1 & 0&\dots&0 \\
0 & a_{1l}&\dots&0\\
\vdots&\vdots&\ddots&\vdots\\
0&0&\dots&a_{l\minus 1\,l}
\end{bmatrix},
 \end{equation}
with an obvious block matrix notation.

The above procedure is not one-to-one, meaning that different sets of parameters may give the same matrix.  As to the question if it is onto,  the simplicity of the above scheme can raise some doubts whether it can  reveal the structure of an arbitrary unitary matrix. Nevertheless, in the next section we shall prove that all unitary matrices can be obtained in this way, provided that all of them possess biunimodular vectors.  

\section {Analysis of $U(n)$}

We have the following basic
\begin{proposition}\label{ana}
A unitary matrix  $A\in U(n)$ has a biunimodular vector $v$ if and only if 
\[
A= \begin{bmatrix}
w_1 & 0&\dots&0 \\
0 & w_2&\dots&0\\
\vdots&\vdots&\ddots&\vdots\\
0&0&\dots&w_n
\end{bmatrix}F_n
\left[ \: 
\begin{array}{*{4}{c}}
1&0&\dots&0\\
\cline{2-4}
0& \multicolumn{1}{|c}{} &  & \multicolumn{1}{c|}{} 
\\
\vdots& \multicolumn{1}{|c}{} &B & \multicolumn{1}{c|}{} 
 \\
0&\multicolumn{1}{|c}{} &  & \multicolumn{1}{c|}{}
  \\
\cline{2-4}

\end{array}
\,\,\right] F_n^*
\begin{bmatrix}
\ov {v_1} & 0&\dots&0 \\
0 & \ov {v_2}&\dots&0\\
\vdots&\vdots&\ddots&\vdots\\
0&0&\dots&\ov {v_n}
\end{bmatrix},
 \]
where $v=(v_1,v_2,\dots,v_n)\in\bt^n$, $Av=w=(w_1,w_2,\dots,w_n)\in\bt^n$, $B$ is some unitary matrix in $U(n-1)$ and $F_n$ is the Fourier transform with the adjoint $F_n^*$.
\end{proposition} 
\begin{prf} We shall use the following notation. For an arbitrary vector $u\in\bc^n$ by $D_u$ we shall denote the diagonal matrix with $u$ on the diagonal. 
 We also consider the stabilizer group of $u$, defined as ${\rm Fix}(u) = \{ S \in U(n): Su = u \}$ and let $\mathbf{1} = \left(1,1,\dots,1\right)$, $\mathbf{e} = \left(1,0,\dots,0\right)$ both in $\bc^n$. 

To conduct the proof we observe, that there exist  unimodular vectors $v$ and $w$ such that $Av=w$  iff $AD_v\mathbf{1}=D_w\mathbf{1}$, what is equivalent  to saying that $D_w^*AD_v\in{\rm Fix}\one$ with $v,w\in\bt^n$. 

The description of the group ${\rm Fix}\one$ is related to the group  ${\rm Fix}(\mathbf{e})$. Indeed, since $F_n\mathbf{e}=\frac{1}{\sqrt n} \mathbf{1}$, it is easy to see that $M\in {\rm Fix}\one$ iff $F_n^*MF_n\in {\rm Fix}(\mathbf{e})$.

Therefore, we see that $A$ has a biunimodular vector $v$ iff $F_n^*D_w^*AD_vF_n\in {\rm Fix}(\mathbf{e})$. Due to the obvious description of 
$ {\rm Fix}(\mathbf{e})$ the proposition follows.
\end{prf}

\begin{remark}\label{r1}
Clearly, if we can find a biunimodular vector $v$ for $A$, then for every $c\in\bt$ the vector $cv$ is also biunimodular for $A$. Therefore, we can concentrate on biunimodular vectors with the leading entry $v_1$ equal to 1. This explains why Proposition~\ref{ana} proves, that every unitary matrix can be obtained via formula (\ref{syn}) iff every unitary matrix has a biunimodular vector.
\end{remark}
\begin{remark} The Fourier transform $F_n$ enters the above scheme as a generic unitary matrix with the property that $F_n\mathbf{e}=\frac{1}{\sqrt n} \mathbf{1}$. We need to clarify that, in the above proposition, $F_n$ can be replaced by any unitary matrix $F$ such that  $F\mathbf{e}=\frac{1}{\sqrt n} \mathbf{1}$. However, clearly the easiest description of such $F$ is given by $F=F_nM$ with $M\in{\rm Fix}(\mathbf{e})$.
\end{remark}

In Section~\ref{num:tra} we shall present an algorithm that  for an explicitly  given unitary matrix allows to search for its biunimodular vectors within a set precision. Therefore,  in practice, a numerically given unitary matrix can be analysed to the point when all parameters  $\{a_{jk}\in\bt:j,k=1,2,\dots, n\}$ are found, that allow to synthesize the matrix back  via formula (\ref{syn}). 

On the other hand, as pointed out by the referees, one has the following crucial

\begin{theorem}\label{IW}{\rm(Idel, Wolf +  Biran, Entov, Polterovich)} Every unitary matrix has a biunimodular vector.
\end{theorem}

This settles that scheme (\ref{syn}) for building unitary matrices gives all of them. The scheme has been discovered independently by  De Vos and De Baerdemacker without the notion of biunimodular vectors. In \cite{VB1} they conjectured that every unitary matrix $U$ enjoys a LAR decomposition: $U=LAR$, where $L$ and $R$ are unitary diagonal matrices and $A$ is a unitary doubly stochastic matrix, meaning that the row sums (and thus columns sums) of $A$ are all equal to 1. This, clearly, is equivalent to saying that  every unitary matrix $U$ has a biunimodular vector and has been proved by Idel and Wolf in \cite{IW}. The proof relies on showing that the action of $U$ on the torus $\bt^n$ defines a Hamiltonian symplectomorphism of the corresponding Clifford torus embeded in the complex projective space. The rest follows from \cite{BEP}, where  Biran, Entov and  Polterovich used Calabi quasimorphisms to prove that the Clifford torus cannot be displaced from itself by a Hamiltonian symplectomorphism.  Idel and 
Wolf point also to the work \cite{Cho} of Cho on Floer cohomology that one can combine with Theorem~\ref{IW} to conclude, that every $n\times n$  unitary matrix has at least $2^{n-1}$ biunimodular vectors.

 In the following section we shall examine  the existence of biunimodular vectors from a few different angles.

\section {Biunimodular problem}

The biunimodular problem is to find (or prove the existence of) a biunimodular vector for a given unitary matrix.
With the notation explained below we can restate this task in equivalent terms.

\begin{theorem}
 \label{char}
Let $A \in U(n)$. The matrix $A$ has a biunimodular vector if and only if either of the following holds
\begin{enumerate}
 \item[(a)] $A \in \mathcal{D} \cdot {\rm Fix}(\mathbf{1}) \cdot \mathcal{D}_1$. 
 \item[(a')] $A' \in \mathcal{D'} \cdot {\rm E}(\mathbf{1}) \cdot \mathcal{D'}$. 
 \item[(b)] $\|  A \|_{\infty \to 1} =n$. 
\item[(c)] $\sup _{u\in\bt^n} |\pi_A(u)|=1$.
\item[(d)] $A(\bt^n)\cap\bt^n\not=\emptyset$.
\end{enumerate}
\end{theorem}
\begin{prf}

(a) Here,  $\mathcal{D} \subseteq U(n)$ denotes the subgroup of diagonal matrices with diagonal entries in $\mathbb{T}$, and $\mathcal{D}_1 \subseteq \mathcal{D}$ is the subgroup of matrices with the first diagonal entry equal to $1$.  ${\rm Fix}(\mathbf{1}) = \{ S \in U(n): S\mathbf{1} = \mathbf{1} \}$, where $\mathbf{1} = \left(1,1,\dots,1\right)\in\bt^n$.  $A \in \mathcal{D} \cdot {\rm Fix}(\mathbf{1}) \cdot \mathcal{D}_1$ means that $A$ can be written as a product of three matrices from the indicated subgroups of $U(n)$. By Proposition~\ref{ana} and Remark~\ref{r1} this is equivalent to the existence of a biunimodular vector for $A$.

(a') Here, $A'=\la A$ with $\la^n=\ov{\det(A)}$, so $A'\in SU(n)$.  $\mathcal{D'}= \mathcal{D}\cap SU(n)$  and  ${\rm E}(\mathbf{1})$ denotes the subgroup of $SU(n)$ that has $\mathbf{1}$ as an eigenvector. Clearly, (a') is a mere translation of (a) into the special unitary setting.

(b) Here, $\|  A \|_{\infty \to 1}=\sup_{\|u\|_{\infty}=1}\|Au\|_1$ is the norm of $A:(\bc^n,\|\cdot\|_\infty)\to(\bc^n,\|\cdot\|_1)$, where  the usual $\ell^p$-norm of $u\in\mathbb{C}^n$ is given by $\|u\|_p=(\sum_{k=1}^n|u_k|^p)^{\frac1p}$ and $\|u\|_\infty=\sup_{k}|u_k|$.

 For every  $u\in\bc^n$ we have $\|u\|_1\le\sqrt n\|u\|_2\le n\|u\|_\infty$. Therefore, for every $A\in U(n)$ we have 
\begin{equation}\label{string}
 \| A u \|_1 \le \sqrt n \| A u \|_2 = \sqrt n \| u\|_2 \le n\|u\|_\infty,
\end{equation}
proving that $\|  A \|_{\infty \to 1}\le n$. Thus, if $A$ has a biunimodular vector $v$, then $\|Av\|_1=n\|v\|_\infty$ showing that $\|  A \|_{\infty \to 1}= n$.

Conversely, if $\| A \|_{\infty \to 1} = n$, compactness of the $\ell^{\infty}$-ball implies the existence of a vector $u \in \mathbb{C}^n$ with $\| u \|_{\infty} = 1$ and $\| A u \|_1 = n\|u\|_\infty$. As the estimate (\ref{string}) shows,  this equality holds if and only if 
$\| A u \|_1 = \sqrt n \| A u \|_2$ and $\| u\|_2 =\sqrt n\|u\|_\infty$, meaning that both $u$ and $Au$ must be a constant multiple of a unimodular vector. Since $\|u\|_\infty=1$ and  $\| A u \|_1 = n$, both $u$ and $Au$ must be unimodular.

(c) Here, $\pi_A(u)=\pi(Au)$, where $\pi(u)=\prod_{k=1}^nu_k$ for $u\in\bc^n$.

Let $A\in U(n)$, $u\in\bt^n$ and $Au=w$. We have
\begin{equation}\label{pi}
|\pi_A(u)|^2=\prod_{k=1}^n|w_k|^2\le \left(\frac1n\|w\|_2^2\right)^n=\left(\frac1n\|u\|_2^2\right)^n=1,
\end{equation}
so $|\pi_A(u)|\le 1$.  Thus, if $A$ has a biunimodular vector $v$, then $|\pi_A(v)|=1$ showing that  $\sup _{u\in\bt^n} |\pi_A(u)|=1$.

Conversely, if $\sup _{u\in\bt^n} |\pi_A(u)|=1$, compactness of $\bt^n$ implies the existence of a vector $u \in\bt^n$ with $ |\pi_A(u)|=1$. In light of (\ref{pi}) this means that $1=\prod_{k=1}^n|w_k|^2=\left(\frac1n\|w\|_2^2\right)^n$, i.e. $w=Au$ must be unimodular.

(d) Clearly, a vector $v$ is biunimodular for $A\in U(n)$ iff $Av\in A(\bt^n)\cap\bt^n$, so the equivalence in this case is trivial. 
\end{prf}

\begin{remark}\label{remchar}
Part (a) of Theorem~\ref{char} is equivalent to  saying that $A \in \mathcal{D} \cdot {\rm Fix}(\mathbf{1}) \cdot \mathcal{D}$. Thus, Theorem~\ref{IW} is equivalent to the decomposition $U(n)=\mathcal{D} \cdot {\rm Fix}(\mathbf{1}) \cdot \mathcal{D}$. Unfortunately, there is no general theory explaining when two subgroups $H$, $H'$ of a Lie group $G$ yield the decomposition $G=H\cdot H'\cdot H$. Moreover, the Cartan decomposition and Kostant decomposition of any semisimple Lie group G, that both are given in this form,  indicate potential difficulties in developing such a general explanation. 

For part (a') included in the above characterization, we would like to notice that $\mathcal{D'}$ is a maximal torus subgroup of a compact simple Lie group $SU(n)$. Thus, Theorem~\ref{IW} is equivalent to the following maximal torus decomposition:
\[
SU(n)= \mathcal{D'} \cdot {\rm E}(\mathbf{1}) \cdot \mathcal{D'},
\]
with $\dim (SU(n))=\dim({\rm E}(\mathbf{1}))+2\dim( \mathcal{D'})$. This allows to conjecture, or ask, whether every compact simple Lie group $G$ admits a  maximal torus decomposition $G=T\cdot H\cdot T$, with a maximal torus $T$ of $G$ and a subgroup $H\subseteq G$ such that $\dim G=\dim H+2\dim T$.

Regarding part (b), as in \cite{GiRz}, an interpolation argument can be used to conclude that $\| A \|_{\infty \to 1} = n$ is in turn equivalent to saying that 
 $\| A \|_{p \to q} = n^{\frac1q-\frac1p}$ in the region $1\le q\le2\le p\le \infty$. This further amplifies the interest in biunimodular vectors, that are  precisely those vectors, where the norm  $\| A \|_{p \to q} = n^{\frac1q-\frac1p}$ is attained, in the indicated range of $p$ and $q$. 

Part (c) of the theorem is equivalent to $\sup _{u\in\bd^n} |\pi_A(u)|=1$, where $\bd=\{z\in\bc: |z|\le 1\}$. Thus, the biunimodular problem can be viewed as an extremum problem for the modulus of a polynomial $\pi_A:\bd^n\to \bd$.

 Part (d) of the characterization provides a deceptively simple geometric interpretation of the problem: Describe the intersection of two tori. This geometric angle has been exploited to prove Theorem~\ref{IW}.
\end{remark}

With the above characterization we close the motivational part of the paper and present our findings on the biunimodular problem.

\section {$U(2)$}\label{secu2}

Obtaining  biunimodular vectors of a given unitary $2\times2$ matrix is an easy task.
\begin{proposition}\label{u2}
Every matrix in $U(2)$ has a biunimodular vector.
\end{proposition}
\begin{prf}
Every matrix $A\in U(2)$ can be written as
\[
A=\begin{bmatrix}
x & y \\
-z\ov y & z\ov x
\end{bmatrix},
\]
with $z\in\bt$ and  $x,y\in\bc$ such that $|x|^2+|y|^2=1$. Finding a biunimodular vector $(1,d)$ for $A$ amounts to solving a system of two equations: 
$|x+yd|=1$ and $|\ov y-\ov x d|=1$. Since  $|x|^2+|y|^2=|d|=1$ the system is equivalent to the equation $\re (\ov x y d)=0$. Therefore, we see that if $xy=0$ then every vector $(1,d)$ with $d\in\bt$ is biunimodular for $A$. While for $xy\not=0$ the matrix $A$ has precisely two biunimodular vectors  with the leading 1 entry:
\begin{equation}\label{sol2}
(1, i e^{i(\al-\be)}), \hskip2cm (1,- i e^{i(\al-\be)}),
\end{equation}
where $\al=\arg x$ and $\be=\arg y$.
\end{prf}

As we indicated in Section \ref{s3}, the existence of biunimodular vectors for every matrix in $U(2)$ allows for writing an alternative formula describing $U(2)$. 
\begin{corollary}\label{u2f}
\begin{equation}\label{u2for}
U(2)=\left\{\frac 12\begin{bmatrix}
a(1+z) & ac(1-z) \\
b(1-z) & bc(1+z)
\end{bmatrix}: a,b,c,z\in\bt\right\}
\end{equation}
\end{corollary}
\begin{prf}
Let $A\in U(2)$. Since $A$ has a biunimodular vector, by Theorem~\ref{char} we have that $A\in  \mathcal{D} \cdot {\rm Fix}(\mathbf{1}) \cdot \mathcal{D}_1$. Clearly, in this case
\[
\fix\one=\left\{\frac12\begin{bmatrix}
1 & 1 \\
1 & -1
\end{bmatrix}
\begin{bmatrix}
1 & 0 \\
0 & z
\end{bmatrix}
\begin{bmatrix}
1 & 1 \\
1 & -1
\end{bmatrix}: z\in\bt\right\}=\left\{\frac12
\begin{bmatrix}
1+z & 1-z \\
1-z & 1+z
\end{bmatrix}: z\in\bt\right\},
\]
and therefore 
\[
A=\frac 12\begin{bmatrix}
a & 0 \\
0 & b
\end{bmatrix}
\begin{bmatrix}
1+z & 1-z \\
1-z & 1+z
\end{bmatrix}
\begin{bmatrix}
1 & 0 \\
0 & c
\end{bmatrix},
\]
with $ a,b,c,z\in\bt$ as desired.
\end{prf}

To close this section we point out that the above observations follow, as well, from writing an arbitrary matrix in $U(2)$ as  
$\begin{bmatrix}
a & 0 \\
0 & b
\end{bmatrix}
\begin{bmatrix}
\cos\vp& \sin\vp \\
-\sin\vp & \cos\vp
\end{bmatrix}
\begin{bmatrix}
1 & 0 \\
0 & c
\end{bmatrix}$ with some $a,b,c\in\bt$, $\vp\in[0,2\pi]$ and noting that 
$\begin{bmatrix}
\cos\vp& \sin\vp \\
-\sin\vp & \cos\vp
\end{bmatrix}=\frac{e^{-i\vp}}2
\begin{bmatrix}
1 & 0 \\
0 & -i
\end{bmatrix}
\begin{bmatrix}
1+z & 1-z \\
1-z & 1+z
\end{bmatrix}
\begin{bmatrix}
1 & 0 \\
0 & i
\end{bmatrix}$ with $z=e^{i2\vp}$ .

\section {$U(3)$}

 Some elementary calculations allow to conclude that 
all orthogonal matrices in $O(3)$ have a biunimodular vector. Another easy case is given in the following.

\begin{example}\label{3zero}
For $A\in U(3)$ with at least one zero entry we shall explain how to exhibit its biunimodular vectors. By interchanging rows and columns of $A$ we can assume that there is a zero entry in the top right corner of $A$. Moreover, we can multiply rows and columns of $A$ by unimodular constants in such a way that the outcome shall have only real entries.
Therefore, we can assume that $A$ is given as
\[
\begin{bmatrix}
\cos \al & \sin\al & 0\\
x\sin \al & -x\cos\al &y\\
y\sin \al & -y\cos\al &-x
\end{bmatrix}
\]
with $\al\in [0,2\pi]$ and $x,y\in\br$ such that $x^2+y^2=1$. A simple check reveals that such a matrix has at least four distinct biunimodular vectors with the leading 1 entry:
\[
(1, i ,e^{i\al}), \hskip1cm (1,- i ,e^{-i\al}), \hskip1cm
(1, i ,-e^{i\al}), \hskip1cm (1,- i ,-e^{-i\al}).
\]
Moreover, if $A$ has only one zero entry, then there are no other biunimodular vectors for $A$. It is also clear, that when $A$ has at least two zero entries, then  at least one of its entries has modulus 1 leading to a trivialization of $A$ and a continuum of biunimodular vectors. 
\end{example}

\begin{remark}
Whenever we talk about the number of biunimodular vectors for a given unitary matrix, we concentrate only on  vectors with the leading entry equal to 1, as it was indicated in Remark~\ref{r1}. The above example explains the structure and the number of biunimodular vectors for $A\in U(3)$ in the case when $A$ has  at least one zero entry. In the case when all entries of $A\in U(3)$ are nonzero,  we have only found examples of matrices with the number of biunimodular vectors equal to 4,5 or 6.
\end{remark}

In order to exhibit the method allowing to count biunimodular vectors of a given matrix $A\in U(3)$, let us consider a vector $u=u_{xy}=(1, e^{ix},e^{iy})$ and the square $Q=[-\pi,\pi]^2$. The idea is to look at three regions $R_j=\{(x,y)\in Q:|Au(j)|\ge1\}$, $j=1,2,3$ and check the intersection of their boundaries.

For the Fourier case $A=F_3$ the corresponding three regions are given in Fig.\ref{fig0}. The points where the boundaries of these regions intersect correspond to biunimodular vectors of $A$.  From the unitarity of $A$ it follows that every intersection point of two such boundaries belongs to the boundary of the third region as well, allowing for a visual count of biunimodular vectors. 
\begin{figure}
                \centering
                \includegraphics[width=6cm]{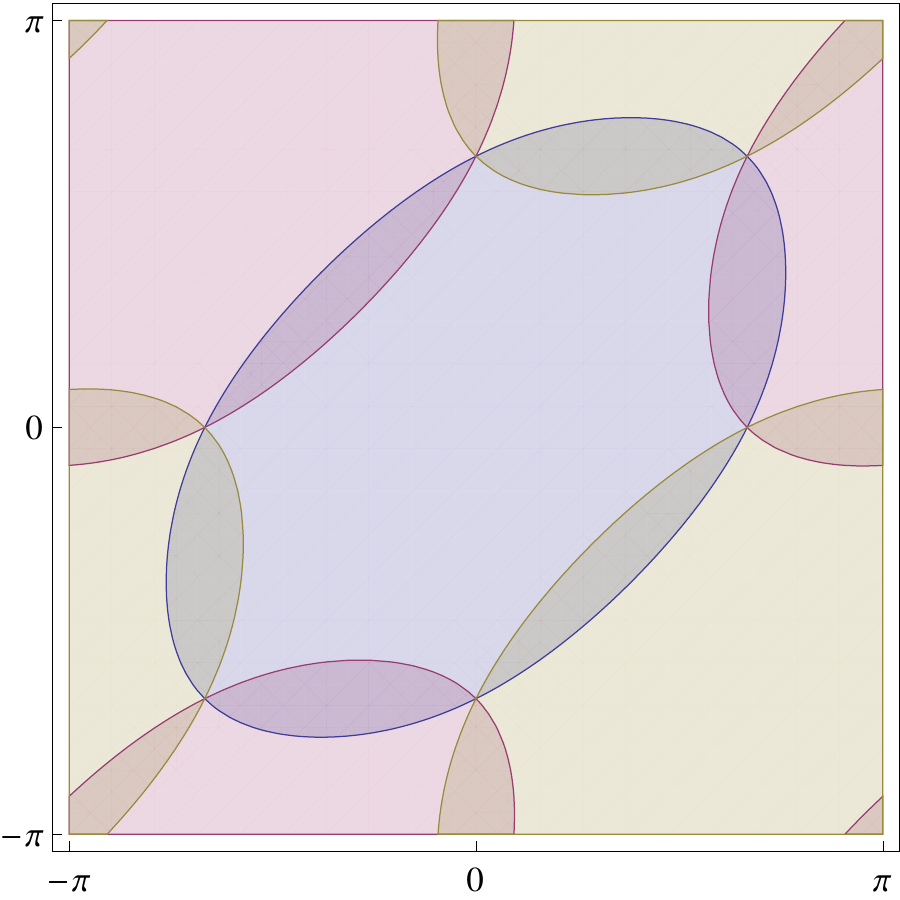}
                \caption{The regions $R_j$, $j=1,2,3$ for the Fourier case $F_3$.}
                \label{fig0}
        \end{figure}
The next figure shows these regions  for three exemplary unitary matrices providing the mentioned count of biunimodular vectors (conducted and plotted with {\it Mathematica}).
\begin{figure}
        \centering
        \begin{subfigure}[b]{0.25\textwidth}
                \centering
                \includegraphics[width=\textwidth]{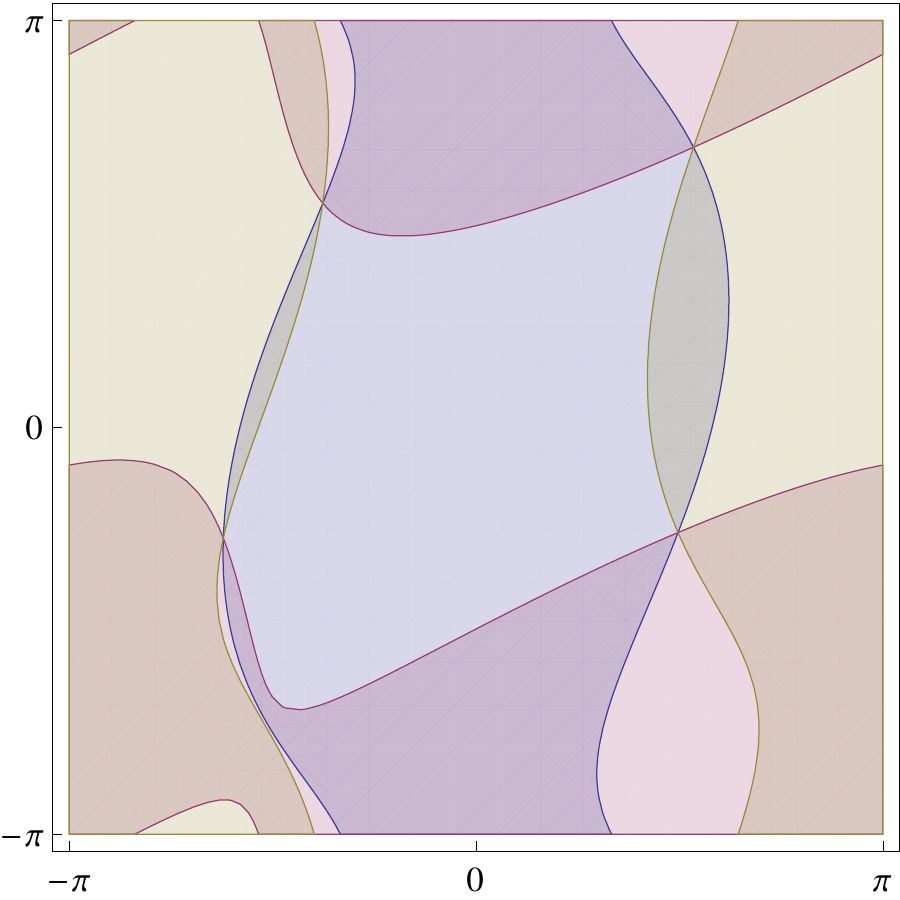}
        \end{subfigure}\quad\quad
      \begin{subfigure}[b]{0.25\textwidth}
                \centering
                \includegraphics[width=\textwidth]{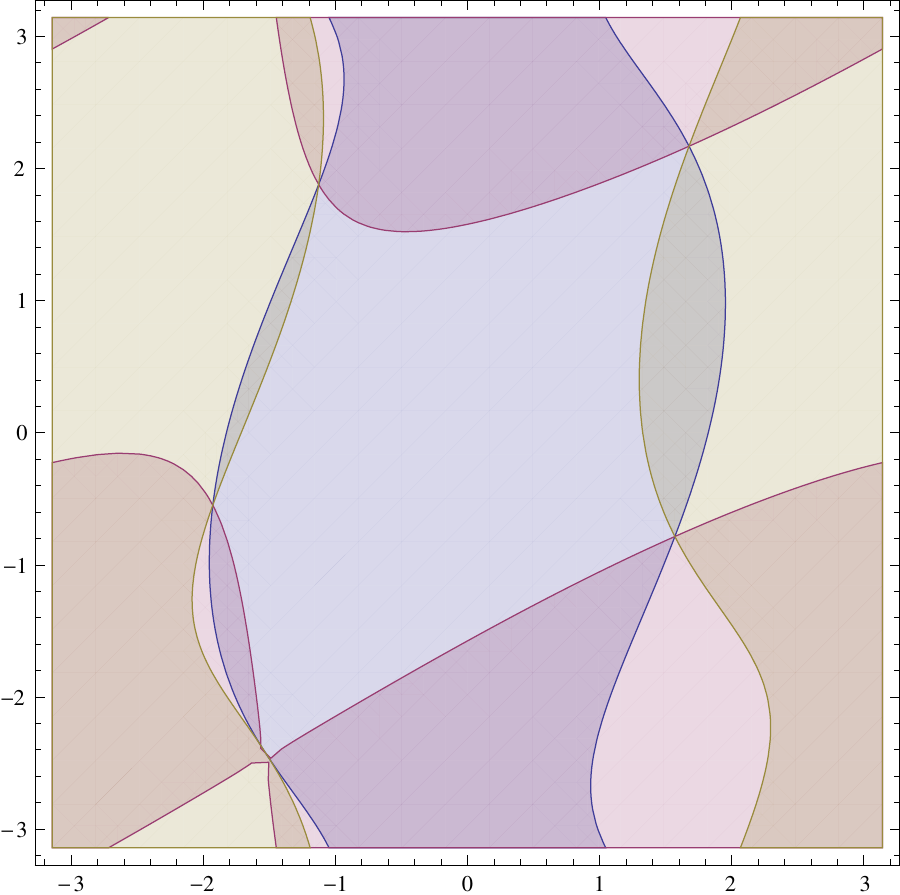}
        \end{subfigure}\quad\quad
        \begin{subfigure}[b]{0.25\textwidth}
                \centering
                \includegraphics[width=\textwidth]{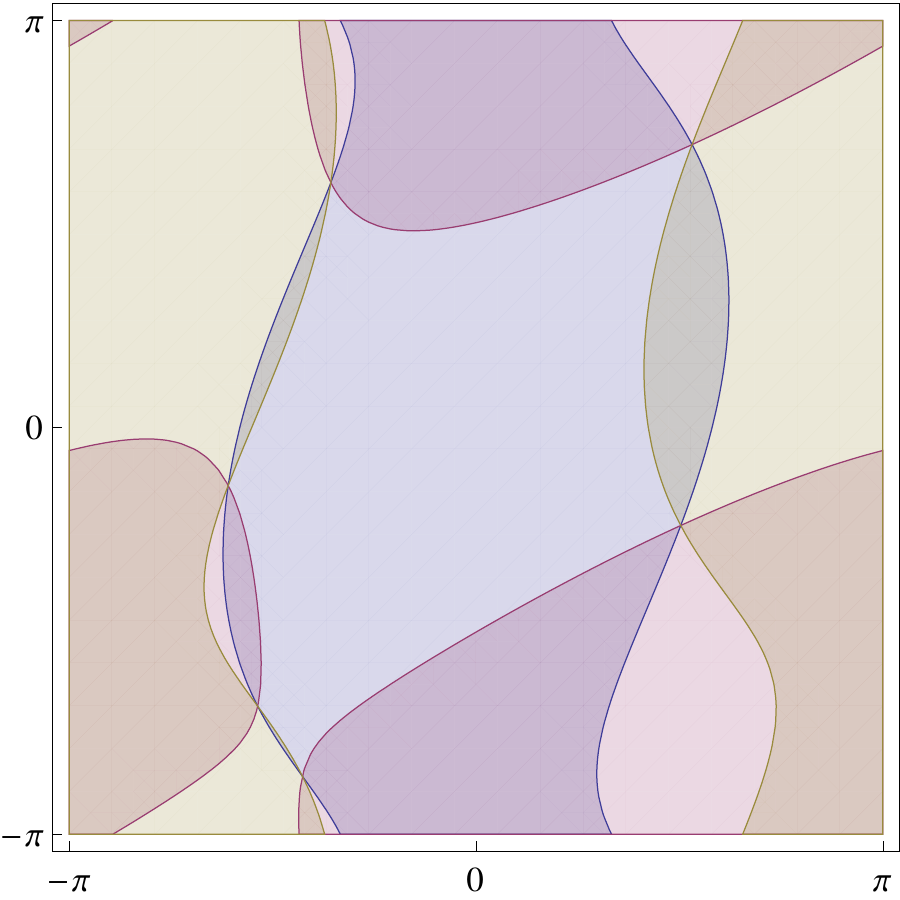}
        \end{subfigure}
        \caption{Counting biunimodular vectors in dimension $n=3$.}\label{fig1}
\end{figure}

 In general, since $A$ is unitary these closed regions cover $Q$, and therefore, any of these regions must intersect with at least one of the two other regions. Unfortunately, it is hard to argue that at least one pair of the regions intersects in such a way that their boundaries also intersect, closing a natural way to prove Theorem~\ref{u3}. This obstacle boils down to the lack of proof for a classification of the regions $R_j$, that essentially can be simplified to 
\[
E_{st}=\{(x,y)\in Q:s\cos(x)+ t\cos(y)+\cos(x-y)\ge0\}, \quad s,t>0
\]
and are hard to study.

A working proof that all $A\in U(3)$ have biunimodular vectors, that we give below, relies on findings of Section~\ref{sec2n}, that are included in Appendix~A.

\begin{theorem}\label{u3}
Every matrix in $U(3)$ has a biunimodular vector.
\end{theorem} 
\begin{prf}
In Appendix~A (Observation~\ref{obsa3}) we prove that 
\[
U(3)=\{\diag(\la_1,\la_2,\la_3)T_{(\al,\be,\ga,z)}\diag(1,\la_4,\la_5): \, \al,\be,\ga\in[0,2\pi]; z,\la_j\in\bt, j=1,\dots,5\},
\]
where
\[
T_{(\al,\be,\ga,z)}=\begin{bmatrix} \cos\al &-\sin\al\sin\ga&\sin\al\cos\ga\\
-\sin\al\sin\be&z\cos\be\cos\ga-\cos\al\sin\be\sin\ga&z\cos\be\sin\ga+\cos\al\sin\be\cos\ga\\
\sin\al\cos\be&z\sin\be\cos\ga+\cos\al\cos\be\sin\ga&z\sin\be\sin\ga-\cos\al\cos\be\cos\ga
\end{bmatrix}.
\]
Thus, in order to prove that every matrix in $U(3)$ has a biunimodular vector, it is enough to show it for the matrix $T_{(\al,\be,\ga,z)}$.

Therefore, we want to find $x,y\in\br$ such that $T_{(\al,\be,\ga,z)}$ applied to the vector $(1,e^{ix},e^{iy})$ yields a unimodular outcome $v\in\bt^3$. The first entry of this outcome is given by  $\cos\al +w\sin\al$ with $w=-e^{ix}\sin\ga+e^{iy}\cos\ga$. By applying $\begin{bmatrix} \cos\al &\sin\al\\ \sin\al &-\cos\al\end{bmatrix}$ to $(1,w)$, we see that $|\cos\al +w\sin\al|=1$ iff $|w|=|\sin\al-w\cos\al|$ iff $\sin\al-w\cos\al=e^{im}w$ for some $m\in\br$. Thus, we can add $m$ as another variable to our problem and require that $\sin\al-w\cos\al=e^{im}w$.

Then, the second entry of $v$ is given as $v_2= uz\cos\be-e^{im}w\sin\be$, where $u=e^{ix}\cos\ga+e^{iy}\sin\ga$. And $v_3=uz\sin\be+e^{im}w\cos\be$. In short,
\begin{equation}\label{twist}
\begin{bmatrix}v_2\\ v_3\end{bmatrix}= \begin{bmatrix} \cos\be &-\sin\be\\ \sin\be &\cos\be\end{bmatrix} \begin{bmatrix} z &0\\0 &e^{im}\end{bmatrix} \begin{bmatrix} \cos\ga &\sin\ga\\ -\sin\ga &\cos\ga\end{bmatrix}\begin{bmatrix}e^{ix}\\ e^{iy}\end{bmatrix}:=U_m \begin{bmatrix}e^{ix}\\ e^{iy}\end{bmatrix}
\end{equation}
and we need to find $x,y,m\in\br$ such that $|v_2|=|v_3|=1$ and $\sin\al-w\cos\al=e^{im}w$, with $w=-e^{ix}\sin\ga+e^{iy}\cos\ga$. Since (\ref{twist}) reads $U_m(e^{ix},e^{iy})\in\bt^2$ with $U_m\in U(2)$, we can recognize the task of finding $x,y,m$ as a biunimodular problem with a certain twist.

By looking at (\ref{sol2}) we can simplify the above problem. Indeed, by Proposition~\ref{u2}, for every $m\in\br$ we can find at least two vectors $(1,\pm ie^{i\varphi_m})$ such that $U_m(1,\pm ie^{i\varphi_m})\in\bt^2$, giving $|v_2|=|v_3|=1$. 
Thus, let us consider $$w_j(m)=-\sin\ga+(-1)^jie^{i\varphi_m}\cos\ga$$ with $j=1,2$ and concentrate on the final  equation $\sin\al-w_j(m)\cos\al=e^{im}w_j(m)$.

 Clearly, if $|\cos \al|=1$, then the matrix  $T_{(\al,\be,\ga,z)}$ is simple enough to have a biunimodular vector, see Remark \ref{3zero}. Hence, we assume $|\cos \al|<1$ and claim that  (\ref{twist}) has a desired solution $x,y,m\in\br$, as long as we can find $m_0\in\br$ and $j_0\in\{1,2\}$ such that 
\begin{equation}\label{sincos}
|w_{j_0}(m_0)|=\left|\frac{\sin \al}{e^{im_0}+\cos\al}\right|
\end{equation}
Indeed, after establishing a solution to (\ref{sincos}), we can find $x_0\in\br$ such that $$e^{ix_0}w_{j_0}(m_0)=\frac{\sin \al}{e^{im_0}+\cos\al},$$
so  $\sin\al-e^{ix_0}w_{j_0}(m_0)\cos\al=e^{im}e^{ix_0}w_{j_0}(m_0)$, leading to the conclusion that $m_0, j_0$ and $x_0$ solve  (\ref{twist}) by setting
$(e^{ix},e^{iy})=e^{ix_0}(1,(-1)^{j_0}ie^{i\varphi_{m_0}})$.

Thus, we are left with solving  (\ref{sincos}). Here, it is important to notice that the biunimodular vectors $(1, ie^{i\varphi_m})$ and  $(1,- ie^{i\varphi_m})$ are orthogonal. Thus, if we define $f_j(m)=|w_j(m)|^2$ for $j=1,2$ and $m\in[0,2\pi]$, then $f_1(m)+f_2(m)=2\|(-\sin\ga,\cos\ga)\|_2^2=2$. Moreover, there is an $m_1\in[0,2\pi]$ such that $e^{im_1}=z$, and then $\varphi_{m_1}=0$ by (\ref{sol2}), so $f_1(m_1)=f_2(m_1)=1$. Finally, let us notice that for $g(m)=\left|\frac{\sin \al}{e^{im}+\cos\al}\right|^2$ one has $\inf g=\left(\frac{|\sin \al|}{1+|\cos\al|}\right)^2\le 1$ and $\sup g=\left(\frac{|\sin \al|}{1-|\cos\al|}\right)^2\ge 1$. 

This allows to finish the proof. Indeed, if $f_1>g$ and $f_2>g$, then $1>g$ contradicting $\sup g\ge 1$. Similarly, if $f_1<g$ and $f_2<g$, then $1<g$ contradicting $\inf g\le 1$. Thus, there is an $m_2\in[0,2\pi]$ such that $g(m_2)$ is between $f_1(m_2)$ and $f_2(m_2)$. However, $g$ can not be strictly between $f_1$ and $f_2$ on $[0,2\pi]$, because $f_1(m_1)=f_2(m_1)$. Therefore, there is an $m_0\in[0,2\pi]$ such that either $g(m_0)=f_1(m_0)$ or $g(m_0)=f_2(m_0)$.
\end{prf}

\begin{remark} Since for any vector $R\in\bc^n$ with $\|R\|_2=1$ the average $\int_{\bt^n}|\<R,u\>|^2du$ is equal to one, it is easy to see, that for such $R$ there is a $u\in\bt^n$, such that $|\<R,u\>|=1$.

 For two orthonormal vectors $R_1,R_2\in\bc^n$ finding a  vector $u\in\bt^n$ such that $|\<R_1,u\>|=|\<R_2,u\>|=1$ seems to be a much harder task, and Theorem~\ref{u3} is equivalent to accomplishing this task with $n=3$. Currently, the only arguments extending this property for $n\ge 4$ rely on Theorem~\ref{IW}. For orthonormal $R_1,R_2\in  \ell^1(\bz)\subset\ell^2(\bz)$ the problem is open.
\end{remark}
As in the $U(2)$ case, the existence of biunimodular vectors for all members of $U(3)$ allows to write a formula for an arbitrary matrix in $U(3)$.

\begin{corollary}\label{u3f}
$A\in U(3)$ if and only if
\[
 A=\frac16\begin{bmatrix}
\la_1 & 0&0 \\
0 & \la_2&0\\
0&0&\la_3
\end{bmatrix}
 \begin{bmatrix}
1 & 1&1 \\
1 &\ov\om & \om \\
1 &\om & \ov\om
\end{bmatrix}
 \begin{bmatrix}
2 & 0&0 \\
0 &a(1+z) & ac(1-z) \\
0&b(1-z) & bc(1+z)
\end{bmatrix}
 \begin{bmatrix}
1 & 1&1 \\
1 &\om &\ov \om \\
1 &\ov\om & \om
\end{bmatrix}
\begin{bmatrix}
1 & 0&0 \\
0 & \la_4&0\\
0&0&\la_5
\end{bmatrix}
\]
for some parameters $a,b,c,z,\la_j\in\bt$, $j=1,\dots,5$ and  $\om=e^{\frac{2\pi i}3}$.
\end{corollary}

\begin{prf}
Let $A\in U(3)$. By Theorems~\ref{u3} and \ref{char} we have that $A\in  \mathcal{D} \cdot {\rm Fix}(\mathbf{1}) \cdot \mathcal{D}_1$. Moreover,  in this case ${\rm Fix}(\mathbf{1})=F_3\cdot{\rm Fix}((1,0,0))\cdot F_3^*$, where $F_3$ is the Fourier transform and 
\[
 {\rm Fix}((1,0,0))=\left\{ \frac 12 \begin{bmatrix}
2 & 0&0 \\
0 &a(1+z) & ac(1-z) \\
0&b(1-z) & bc(1+z)
\end{bmatrix} :a,b,c,z\in\bt\right\}
\]
due to the description of $U(2)$ provided in Corollary~\ref{u2f}. Since $F_3=\frac1{\sqrt3} \begin{bmatrix}
1 & 1&1 \\
1 &\ov\om & \om \\
1 &\om & \ov\om
\end{bmatrix}$\break
we can finish the proof.
\end{prf}

Moreover, as a byproduct from the proof of Theorem~\ref{u3} we obtain the following description of $U(3)$.

\begin{corollary}\label{euler}
\begin{align*}
U(3)&=\{\diag(\la_1,\la_2,\la_3) X_\be Y^z_\al X_\ga \diag(1,\la_4,\la_5): \, \al,\be,\ga\in[0,2\pi]; z,\la_j\in\bt, j=1,\dots,5\}\\
&=\{\diag(\la_1,\la_2,\la_3) Z_\be Y^z_\al X_\ga \diag(1,\la_4,\la_5): \, \al,\be,\ga\in[0,2\pi]; z,\la_j\in\bt, j=1,\dots,5\},
\end{align*}
where for a given $\varphi\in[0,2\pi]$
\[
X_\vp=\begin{bmatrix}
1 & 0&0 \\
0 &\cos\vp & -\sin\vp \\
0&\sin\vp & \cos\vp
\end{bmatrix},\quad
Y^z_\vp=\begin{bmatrix}
\cos\vp& 0&-\sin\vp\\
0 &z& 0\\
\sin\vp&0 & \cos\vp
\end{bmatrix},\quad
Z_\vp=\begin{bmatrix}
\cos\vp & -\sin\vp&0 \\
\sin\vp & \cos\vp&0\\
0&0&1
\end{bmatrix}.
\]
\end{corollary}
\begin{prf}
This follows from Observation~\ref{obsa3} used in the proof of Theorem~\ref{u3} and the matrix factorization:
\[
T_{(\al,\be,\ga,z)}=\begin{bmatrix}
1 & 0&0 \\
0 &\cos\be & -\sin\be \\
0&\sin\be & \cos\be
\end{bmatrix}
\begin{bmatrix}
\cos\al & 0&-\sin\al\\
0 &z& 0\\
\sin\al&0 & \cos\al
\end{bmatrix}
\begin{bmatrix}
1 & 0&0 \\
0 &\cos\ga & \sin\ga \\
0&\sin\ga & -\cos\ga
\end{bmatrix},
\]
that leads to the provided description of $U(3)$ via easy matrix manipulations.
\end{prf}

The above description of $U(3)$ can be compared with the standard Euler-Tait-Bryan decomposition of $O(3)$:
\[
O(3)=\{ X_\be Y^z_\al X_\ga: \, \al,\be,\ga\in[0,2\pi]; z=\pm1\}=\{ Z_\be Y^z_\al X_\ga: \, \al,\be,\ga\in[0,2\pi]; z=\pm1\}.
\]

\section {$U(4)$ and beyond}

The study conducted in the two previous sections allows to surmise that, excluding some trivial cases, the number of biunimodular vectors for $A\in U(n)$ is finite when $n\le 3$. The classic example of the Fourier transform $F_4$ indicates a crucial change in the amount of biunimodular vectors, when $n\ge 4$. To understand this phenomenon it is enough to notice that $(1,0,1,0)$ and $(0,1,0,-1)$ are eigenvectors of $F_4$. Thus, $(1,z,1,-z)$ is biunimodular for $F_4$ for all $z\in\bt$. An investigation of the remaining biunimodular vectors for $F_4$, that are given as $(1,z,-1,z)$, allows to see a pattern that guarantees a continuum of biunimodular vectors.

In order to view this pattern, for $u\in\bc^n$ we define its support as $\supp(u)=\{k\in\{1,2,\dots,n\}: u_k\not=0\}$ and consider the following.
\begin{observ}\label{obs} Let $A\in U(n)$. If there are $u,w\in\bc^n$ such that each of the vectors $|u|,|w|,|Au|,|Aw|$ is equal to 1 on its support and
 both disjoint unions $\supp(u)\cup \supp(w)=\supp(Au)\cup \supp(Aw)$ cover $\{1,2,\dots,n\}$, then $u+zw$ is biunimodular for $A$ for all $z\in\bt$.
\end{observ}

The above observation allows to construct matrices in $U(n)$ with all nonzero entries having a continuum of biunimodular vectors, whenever $n\ge4$. It also allows for a further appreciation of Haagerup's result stated in Theorem~\ref{hag}. It is not clear, however, whether for all unitary matrices with a continuum of biunimodular vectors one can find vectors $u$ and $w$ as in  Observation~\ref{obs}. Even if this would be true, pointing to a conclusion that for any unitary matrix the set of its biunimodular vectors is a (finite) union of tori, it still would not tell, whether the set can be empty or not, underscoring the importance of Theorem~\ref{IW}.

In the next section, we provide an alternative way of building $U(2^n)$, that is based solely on the biunimodular structure of $U(2)$. After that, we conduct the study of biunimodular vectors in the general $U(n)$ case from the Lie groups perspective. Finally, we present a numerical treatment of the biunimodular problem with certain applications to the classical Fourier case.

\section{$U(2^n)$}\label{sec2n}

In this section we show, that a further generalization of the biunimodular problem allows to build all matrices in $U(2n)$ from matrices in $U(n)$.

In order to achieve this goal we employ the formula for $U(2)$ given in Section~\ref{secu2}:
\begin{equation}\label{u2f2}
U(2)=\left\{\frac 12\begin{bmatrix}
a & 0 \\
0 & b
\end{bmatrix}
\begin{bmatrix}
1+z & 1-z \\
1-z & 1+z
\end{bmatrix}
\begin{bmatrix}
1 & 0 \\
0 & c
\end{bmatrix}\,:  a,b,c,z\in\bt \right\}
\end{equation}
and replace the unimodular numbers  $a,b,c,z\in\bt$ by unitary matrices $A,B,C,Z\in U(n)$.
\begin{theorem}\label{bimat} For every $n\in\bn$
\begin{equation}\label{u2n}
U(2n)=\left\{\frac 12\begin{bmatrix}
A & 0 \\
0 &B
\end{bmatrix}
\begin{bmatrix}
I+Z & I-Z \\
I-Z & I+Z
\end{bmatrix}
\begin{bmatrix}
I & 0 \\
0 & C
\end{bmatrix}\,:  A,B,C,Z\in U(n) \right\},
\end{equation}
where $I$ stands for the identity matrix in $U(n)$.
\end{theorem}
\noindent

The above theorem yields immediately a map from $\bt^{4^n}$ onto $U(2^n)$ via an obvious inductive procedure starting from the trivial map from $\bt$ onto $U(1)$.  In the first step of this procedure we recover (\ref{u2f2}). In the next step we get a formula for $U(4)$ that is enclosed in Appendix~A.

Before we prove the theorem, we need a slim primary on block matrices. Let $M(n)$ denote the set of all $n\times n$ matrices with complex entries. An arbitrary matrix in $M(2n)$ can be written in a block form as $\begin{bmatrix}
A & B \\
C & D
\end{bmatrix}$
with $A,B,C,D\in M(n)$.

 Clearly, we have that  $\begin{bmatrix}A & B \\C & D\end{bmatrix}\begin{bmatrix}A' & B' \\C' & D'\end{bmatrix}=\begin{bmatrix}AA'+BC' & AB'+BD' \\CA'+DC' & CB'+DD'\end{bmatrix}$ and $\begin{bmatrix}A & B \\C & D\end{bmatrix}^*=\begin{bmatrix}A^* & C^* \\B^* & D^*\end{bmatrix}$.
 Moreover, every $2n\times n$ matrix with complex entries can be written as $\begin{bmatrix}X \\Y\end{bmatrix}$ with $X,Y\in M(n)$ and 
$\begin{bmatrix}A & B \\C & D\end{bmatrix}\begin{bmatrix}X \\Y\end{bmatrix}=\begin{bmatrix}AX+BY \\CX+DY\end{bmatrix}$.
\begin{lemma}\label{bimal}
Let $\frac1{\sqrt 2} \begin{bmatrix}A & B \\C & D\end{bmatrix}\in U(2n)$ with $A\in U(n)$ and $B,C,D\in M(n)$. Then  $B,C,D\in U(n)$.
\end{lemma}
\begin{proof}
Let $U=\frac1{\sqrt 2} \begin{bmatrix}A & B \\C & D\end{bmatrix}$ and $I\in U(n)$ be the identity matrix. Since $UU^*= \begin{bmatrix}I & 0 \\0 & I\end{bmatrix}$ we get that $BB^*=I$. Considering $U^*U= \begin{bmatrix}I & 0 \\0 & I\end{bmatrix}$ yields that $C^*C=I$ and $D^*D=I$.
\end{proof}
\begin{proof}[Proof of Theorem \ref{bimat}]
We can recognize the task as a special sort of a biunimodular problem. For an arbitrary unitary matrix $U\in U(2n)$ we need to find a unitary matrix $C\in U(n)$ such that $U\begin{bmatrix}I \\C^*\end{bmatrix}$ is ``unimodular'', in the sense that  
\begin{equation}\label{bima}
U\begin{bmatrix}I \\C^*\end{bmatrix}= \begin{bmatrix}A \\B\end{bmatrix} \text{ with } A,B,C\in U(n).  
\end{equation}
(The general biunimodular problem of finding $A,B,C,D\in U(n)$ such that  $U\begin{bmatrix}D^* \\C^*\end{bmatrix}= \begin{bmatrix}A \\B\end{bmatrix}$ collapses to the latter since  $U\begin{bmatrix}D^* \\C^*\end{bmatrix}= \begin{bmatrix}A \\B\end{bmatrix}$ iff  $U\begin{bmatrix}I \\C^*D\end{bmatrix}= \begin{bmatrix}AD \\BD\end{bmatrix}$.) Therefore, the goal is to repeat the arguments used in Section~\ref{secu2}, carefully
replacing the unimodular numbers in $\bt$ by unitary matrices in $U(n)$.

We start by noting that solving (\ref{bima}) amounts to saying that $ \begin{bmatrix}A^* & 0 \\0 & B^*\end{bmatrix}U \begin{bmatrix}I & 0 \\0 & C^*\end{bmatrix} \begin{bmatrix}I \\I\end{bmatrix}= \begin{bmatrix}I \\I\end{bmatrix}$, i.e. that $V= \begin{bmatrix}A^* & 0 \\0 & B^*\end{bmatrix}U \begin{bmatrix}I & 0 \\0 & C^*\end{bmatrix}$ belongs to $\fix \begin{bmatrix}\scriptstyle I \\\scriptstyle I\end{bmatrix}=\left\{V\in U(2n)\,:V \begin{bmatrix}I \\I\end{bmatrix}= \begin{bmatrix}I \\I\end{bmatrix}\right\}$. The description of the group $\fix \begin{bmatrix}\scriptstyle I \\\scriptstyle I\end{bmatrix}$ is again related to  $\fix \begin{bmatrix}\scriptstyle I \\\scriptstyle 0\end{bmatrix}=\left\{ \begin{bmatrix}I&0 \\0&Z\end{bmatrix}\,:Z\in U(n)\right\}$ since $\FF=\frac1{\sqrt 2} \begin{bmatrix}I & I \\I & -I\end{bmatrix}\in U(2n)$ and $\FF\begin{bmatrix}I \\0\end{bmatrix}=\frac1{\sqrt 2}\begin{bmatrix}I \\I\end{bmatrix}$. Thus, $V\in \fix \begin{bmatrix}\scriptstyle I \\\scriptstyle I\end{bmatrix}$ 
iff $\FF V\FF \in \fix \begin{bmatrix}\scriptstyle I \\\scriptstyle 0\end{bmatrix}$, being in turn equivalent to $V=\FF  \begin{bmatrix}I&0 \\0&Z\end{bmatrix} \FF=\frac12 \begin{bmatrix}
I+Z & I-Z \\
I-Z & I+Z
\end{bmatrix}$ with some $Z\in U(n)$.

In this way we get that 
$$
U=\frac 12\begin{bmatrix}
A & 0 \\
0 &B
\end{bmatrix}
\begin{bmatrix}
I+Z & I-Z \\
I-Z & I+Z
\end{bmatrix}
\begin{bmatrix}
I & 0 \\
0 & C
\end{bmatrix},
$$
for some $A,B,C\in U(n)$ iff  (\ref{bima}) holds.

In order to solve the biunimodular problem  (\ref{bima}) we write $U= \begin{bmatrix}X & Y \\X' & Y'\end{bmatrix}$ with $X,Y,X',Y'\in M(n)$ and apply the singular value decomposition (SVD) to $X$ and $Y$.

Recall that for an arbitrary matrix $M\in M(n)$ there are unitary matrices $W,W'\in U(n)$ and a diagonal matrix with non-negative entries $\Sigma\in M(n)$ such that $M=W'\Sigma W^*$. This leads to the polar decomposition $M=S_MU_M$ with $S_M=W'\Sigma (W')^*$ - a positive semi-definite matrix and a unitary matrix $U_M=W'W^*\in U(n)$. 

In short, we use polar decomposition $X=S_XU_X$, $Y=S_YU_Y$ to write
$$
U= \begin{bmatrix}S_XU_X & S_Y U_Y\\X' & Y'\end{bmatrix}
$$
and claim that 
$$
U \begin{bmatrix}I \\ iU_Y^*U_X\end{bmatrix}= \begin{bmatrix}A \\B\end{bmatrix}
$$
with $A,B\in U(n)$ as desired, so $C^*= iU_Y^*U_X$ solves problem  (\ref{bima}). To see this, we check that $A=(S_X+iS_Y)U_X$ and $AA^*=S_X^2+S_Y^2+i(S_YS_X-S_XS_Y)$. Since $UU^*= \begin{bmatrix}I & 0 \\0 & I\end{bmatrix}$ we get that $S_X^2+S_Y^2=I$, so $S_X^2$ and $S_Y^2$ commute. For positive semi-definite matrices $S_X$ and $S_Y$ this means that they commute as well and, therefore, $A\in U(n)$.

To finish the proof we make the final observation that $A\in U(n)$ already implies that $B\in U(n)$. Indeed, we can extend the matrix $\frac1{\sqrt 2} \begin{bmatrix}I  \\ C^*\end{bmatrix}$ to a unitary matrix $\frac1{\sqrt 2} \begin{bmatrix}I & C \\C^* & -I\end{bmatrix}\in U(2n)$ and see that
$ \frac1{\sqrt 2}U \begin{bmatrix}I & C \\C^* & -I\end{bmatrix}=\frac1{\sqrt 2} \begin{bmatrix}A & X'' \\B & Y''\end{bmatrix}\in U(2n)$. Since $A$ is unitary, Lemma~\ref{bimal} assures that $B\in U(n)$.

Therefore, we have proved that every $U\in U(2n)$ can be written in the way prescribed by (\ref{u2n}).
On the other hand, we have already checked that $\FF  \begin{bmatrix}I&0 \\0&Z\end{bmatrix} \FF=\frac12 \begin{bmatrix}
I+Z & I-Z \\
I-Z & I+Z
\end{bmatrix}\in U(2n)$ for an arbitrary $Z\in U(n)$, so formula (\ref{u2n}) holds.
\end{proof}

\section{$U(n)$}
\label {sect:factor}

We are interested in finding the biunimodular vectors associated to a matrix $A \in U(n)$ for any $n \in \mathbb{N}$. As it was pointed out in Remark~\ref{r1} we can concentrate on biunimodular vectors with the leading entry equal to 1. Let us denote the collection of such vectors  as $\bi(A)$, so
\[
\bi(A)=\{(v_1,v_2,\dots,v_n)\in \bt^n\cap A^{*}(\bt^n):v_1=1\}.
\]

Moreover, let us make the following formal definition:
\[
\mathcal{B}_n = \{ A \in U(n) : {\rm Bi}(A) \not= \emptyset \}.
\]
Clearly, $\mathcal{B}_n$ is the set of  $n\times n$ unitary matrices that possess biunimodular vectors and Theorem~\ref{IW} asserts that $\mathcal{B}_n=U(n)$. While it would be hard to reproduce this result, one can check what information about $\mathcal{B}_n$ and  $\bi(A)$ can be drawn by using basic tools of Lie groups theory.

\subsection{The main correspondence}

The main tool for understanding $\mathcal{B}_n$ and $\bi(A)$ will be the factorization conducted in Proposition~\ref{ana}. As we noticed in Theorem~\ref{char}(a) the proposition implies that $A\in\mathcal{B}_n$ iff $A\in \mathcal{D} \cdot {\rm Fix}(\mathbf{1}) \cdot \mathcal{D}_1$, where
 $\mathcal{D} \subseteq U(n)$ denotes the subgroup of diagonal matrices with diagonal entries in $\mathbb{T}$,  $\mathcal{D}_1 \subseteq \mathcal{D}$ is the subgroup of matrices with the first diagonal entry equal to $1$ and ${\rm Fix}(\mathbf{1}) = \{ S \in U(n): S\mathbf{1} = \mathbf{1} \}$ with $\mathbf{1} = \left(1,1,\dots,1\right) \in \mathbb{C}^n$ is also a subgroup of $U(n)$. This allows to consider the following map 
\[\Phi: \mathcal{D} \times {\rm Fix}(\mathbf{1}) \times \mathcal{D}_1 \to U(n)\]
 \[\Phi(D_1,S,D_2) = D_1 S D_2\] containing the entire information concerning the existence of biunimodular vectors. 

Indeed,  the matrices possessing biunimodular vectors are given as the image of $\Phi$, what already provides some information on $\mathcal{B}_n$.

\begin{proposition}\label{prop isi}
 $\mathcal{B}_n = \mathcal{D} \cdot {\rm Fix}(\mathbf{1}) \cdot \mathcal{D}_1=\Phi( \mathcal{D} \times {\rm Fix}(\mathbf{1}) \times \mathcal{D}_1 )$. In particular, $\mathcal{B}_n$ is a closed, pathwise connected subset of $U(n)$. 
\end{proposition}
\begin{prf}
 As we mentioned before, the equality  $\mathcal{B}_n = \mathcal{D} \cdot {\rm Fix}(\mathbf{1}) \cdot \mathcal{D}_1$ follows immediately from Theorem~\ref{char}(a). Thus, by definition,  $\mathcal{B}_n =\Phi( \mathcal{D} \times {\rm Fix}(\mathbf{1}) \times \mathcal{D}_1 )$. Since the subgroups $\mathcal{D},\, {\rm Fix}(\mathbf{1}), \,\mathcal{D}_1$ are compact and pathwise connected, $\mathcal{B}_n$, which is the continuous image of their cartesian product, has the same properties. 
\end{prf}

Moreover,  the preimage $\Phi^{-1}(A)$ bijectively corresponds to $\bi(A)$.
\begin{proposition}\label{prop isi2}
There is a continuous bijection between $\bi(A)$ and  $\Phi^{-1}(A)$.
\end{proposition}
\begin{prf} With the notation that $D_u$ is a diagonal matrix with entries on the diagonal given by $u\in\bc^n$, the bijection can be described as follows. The preimage $\Phi^{-1}(A)$ consists of all triplets $(D_1,S,D_2)\in  \mathcal{D} \times {\rm Fix}(\mathbf{1}) \times \mathcal{D}_1 $ such that $A=D_1SD_2$. In light of Proposition~\ref{ana} this means that $D_2=D_{\ov v}$ with $v\in\bi(A)$, $D_1=D_{Av}$, and $S=D_{{Av}}^{-1}AD_{ \ov{v}}^{-1}=D_{\ov{Av}}AD_{ v}$. Therefore, the projection $(D_1,S,D_2)\to D_2$ and the map $D_u\to \ov u$ yield the continuous bijection from   
 $\Phi^{-1}(A)$ onto $\bi(A)$, that is given by  $(D_1,S,D_u)\to \ov u$.
\end{prf}
\begin{remark}
  $U(n)$ has (real) dimension $n^2$, whereas ${\rm Fix}(\mathbf{1}) \cong U(n-1)$ is of dimension $(n-1)^2$. $\mathcal{D}$ has dimension $n$ and $\mathcal{D}_1$ has dimension $n-1$, hence $\Phi$ is a mapping between two manifolds of (real) dimension $n^2$. Clearly, $\Phi$ is smooth.
\end{remark}

\subsection{The structure of $\mathcal{B}_n$}

In this subsection we show, that  for every $n\in\bn$ the set $\mathcal{B}_n$ contains a nonempty open subset of $U(n)$.
This is done by calculating the Jacobian of the map $\Phi: \mathcal{D} \times {\rm Fix}(\mathbf{1}) \times \mathcal{D}_1 \to U(n)$ and  proving that at a certain point it has a full rank. This allows to apply the inverse function theorem to draw the above conclusion.

In order to proceed with the calculations we provide some basic facts from the theory of Lie groups and establish the necessary notation.

We consider $U(n)$ and its subgroups as (real) submanifolds of the space $\mathfrak{gl}(n,\mathbb{C})$ of arbitrary matrices. Given a closed subgroup $H \subseteq GL(n,\mathbb{C})$ and $g \in H$, the tangent space $T_g H$ is the space of all tangent vectors of curves in $H$ going through $g$, conveniently viewed as elements of the ambient space $\mathfrak{gl}(n,\mathbb{C})$ containing $H$. The tangent space $T_e H$ is a Lie-subalgebra of $\mathfrak{gl}(n,\mathbb{C})$, i.e. it is closed under the Lie bracket
$[X,Y] = XY-YX$, and it is referred to as the Lie algebra $\mathfrak{h}$ of $H$. Here, $e$ is the neutral element, that is, the identity matrix that we shall denote by $I$.  An alternative description of $\mathfrak{h}$ is obtained via the matrix exponential, defined by
\[
 \exp(X) = \sum_{k =0}^\infty \frac{X^k}{k!}~.
\] It is easy to verify that $\frac{d}{dt}\left(t \mapsto \exp(tX) \right)|_{t=0} = X$, and $\mathfrak{h}$ consists of all matrices $X$ such that $\exp(tX) \in H$, for all $t \in \mathbb{R}$.

 Since multiplication (on the left or on the right) with a group element $g$  is a diffeomorphism on $H$, that is a matrix group, one obtains that
$$T_g H = g \cdot\mathfrak{h} = \{ g X : X \in\mathfrak{h} \}=\mathfrak{h}\cdot g=\{  Xg : X \in\mathfrak{h} \}.$$

We let $\mathfrak{u},\mathfrak{d},\mathfrak{d}_1, \mathfrak{f}$ denote the Lie algebras of $U(n), \mathcal{D},\mathcal{D}_1,{\rm Fix}(\mathbf{1})$, respectively. Then
\[
 \mathfrak{u} = \{ X \in\mathfrak{gl}(n,\mathbb{C}) ~:~X^* = - X \}~,~
\] $\mathfrak{d}$ consists of all diagonal matrices with purely imaginary entries, and $\mathfrak{d}_1 \subset \mathfrak{d}$ is the subspace of matrices with vanishing upper left corner. Differentiating the equality $\exp(tX) \mathbf{1} = \mathbf{1}$ at $t=0$ yields the following description of $\mathfrak{f}$: 
\[
 \mathfrak{f} = \{ X \in \mathfrak{u} : X \mathbf{1} = 0 \}~,
\] i.e. the elements of $\mathfrak{f}$ are characterized by the fact that the sum over rows are zero. We also note the simple fact that $\mathfrak{d} \cap \mathfrak{f} = \{ 0 \}$.

Finally, recall the definition of Jacobians.  If $F: X \to Y$ is a smooth map between manifolds $X$ and $Y$, the Jacobian of $F$ at $x \in X$ is the linear map $dF(x): T_x X \to T_{F(x)} Y$ defined as follows: For any $v \in T_x X$ pick a curve $\gamma : (-\epsilon,\epsilon) \to X$ with $\gamma(0) = x$, $\dot{\gamma}(0) = v$. Define the curve $\tilde{\gamma} = F \circ \gamma: (-\epsilon,\epsilon) \to Y$ (by definition $\tilde{\gamma}(0) = F(x)$). Then $dF(x)(v) = \dot{\tilde{\gamma}}(0)$.

To calculate the Jacobian  of  $\Phi: \mathcal{D} \times {\rm Fix}(\mathbf{1}) \times \mathcal{D}_1 \to U(n)$ we shall use the following
\begin{lemma}  \label{lem:Jac_mult}
 Let $H_1,\,H_2$ be closed subgroups of $U(n)$ and $m:  H_1 \times H_2 \to U(n)$, $m(g,h) = g\cdot h$. Then, for all $g, h \in U(n)$, the Jacobian of $m$ at $(g,h)$ is given by 
\[
 dm(g,h) (X,Y) = g Y + Xh~, ~\forall X \in T_g H_1, Y \in T_h H_2~.
\]
\end{lemma}
\begin{prf}
 Let $\gamma_i: (-\epsilon,\epsilon) \to H_i$ denote smooth curves with 
\[ \gamma_1(0) = g, \dot{\gamma}_1(0) =X \mbox{ and } \gamma_2(0) = h, \dot{\gamma}_2(0) = Y ~.\]By a standard reasoning, the product rule for scalar-valued functions easily extends to matrix-valued functions implying that 
\[
 \frac{d(\gamma_1(t) \gamma_2(t))}{dt}(0) = \dot{\gamma}_1(0) \gamma_2(0) + \gamma_1(0) \dot{\gamma}_2(0) = 
gY + Xh~.
\]
\end{prf}

Using this lemma and the chain rule, we obtain a rather transparent description of the Jacobian $d\Phi(D_1,S,D_2)$:
\begin{lemma} \label{lem:Jac_Phi}
Let $(D_1,S,D_2) \in \mathcal{D} \times {\rm Fix}(\mathbf{1}) \times \mathcal{D}_1$. Then 
\[
 d\Phi(D_1,S,D_2) : T_{D_1} \mathcal{D} \times T_{S} ({\rm Fix}(\mathbf{1})) \times T_{D_2}  \mathcal{D}_1 \to T_{D_1 S D_2} U(n)  
\] acts via
\[
 d \Phi(D_1,S,D_2) (X,Y,Z)= D_1 (SZ+YD_2) +XSD_2
\]
\end{lemma}
\begin{prf}
 Since we can factor
\[
 \Phi: (D_1,S,D_2) \mapsto (D_1, m(S,D_2)) \mapsto m(D_1,m(S,D_2))~,
\]
applying the previous lemma twice and using the chain rule yields
\begin{eqnarray*}
 d\Phi(D_1,S,D_2) (X,Y,Z) & = &  \left[ dm(D_1,S D_2) \circ (I, dm(S,D_2)) \right] (X,Y,Z) \\
 & =&  dm(D_1,S D_2)(X,SZ+YD_2) = 
D_1 (SZ+YD_2) +XSD_2~. 
\end{eqnarray*}
\end{prf}

Recall that we are particularly interested in the rank of $d\Phi$, that is, the real dimension of the image of $d\Phi$. The following lemma translates this to a more tractable problem in linear algebra: 
\begin{lemma} \label{lem:regular_rank}
Let $S \in {\rm Fix}(\mathbf{1})$ be given, and $m \in \mathbb{N}$. Then 
the following are equivalent:
\begin{enumerate}
 \item[(a)] For all (any) $(D_1,D_2) \in \mathcal{D} \times \mathcal{D}_1$ the Jacobian $d\Phi(D_1,S,D_2)$ has rank $m$.
 \item[(b)] $d\Phi({\rm I},S,{\rm I})$ has rank $m$. 
 \item[(c)] The linear map 
\[ C_S : \mathfrak{d} \times \mathfrak{f} \times \mathfrak{d}_1 \to \mathfrak{u}~,~ (X,Y,Z) \mapsto  X + Y + SZS^* \]
has rank $m$. 
\end{enumerate}
\end{lemma}
\begin{prf}
 For the equivalence (a) $\Leftrightarrow$ (b) we observe that $\Phi(C D_1,S, D_2 C') = C \Phi( D_1,S,D_2)C'$ for all $C \in \mathcal{D}$, and $C' \in \mathcal{D}_1$. Thus, by the chain rule, $d\Phi(C D_1,S,D_2 C') =  C d\Phi(D_1,S,D_2) C'$,  what yields the equivalence. 

To see that (b) $\Leftrightarrow$ (c) we note, that by the previous lemma, $d\Phi(I,S,I) (X,Y',Z) =XS+Y'+SZ$, for all $X \in \mathfrak{d}, Y' \in T_S ({\rm Fix}(\mathbf{1})), Z \in \mathfrak{d}_1$. Since $S$ is invertible, the rank of $d\Phi(I,S,I)$ equals the rank of $R_{S^*} \circ d\Phi(I,S,I)$, where $R_{S^*}$ denotes right multiplication with $S^*$. Clearly, $R_{S^*} \circ d\Phi(I,S,I)(X,Y',Z)= X + Y'S^* + SZS^*$ and this map has the same rank as the map $C_S$ from (c), since $Y=  Y'S^*\in \mathfrak{f}$ iff $ Y' \in T_S ({\rm Fix}(\mathbf{1}))$.
\end{prf}

We next translate the condition from part (c) of the previous lemma to a simpler one. In the following lemma and its proof, the real and imaginary parts of complex vectors and matrices are computed componentwise. In particular, given a complex-valued matrix $U$, ${\rm Im}(U)$ is obtained by taking the imaginary parts of the individual entries of $U$ and the rank of  ${\rm Im}(U)$ is its rank as a mapping on $\br^n$. Also note that throughout the proof of the lemma, ``linear'' means ``$\mathbb{R}$-linear''; and matrices are identified with the linear map they induce. 
\begin{lemma}  \label{lem:rank_diff_simple}
Let $S \in {\rm Fix}(\mathbf{1})$, then
\begin{equation} \label{eqn:rank_formula}
 {\rm rank}(d \Phi(I,S,I)) = (n-1)^2 + n + {\rm rank}({\rm Im}(S))~.
\end{equation}
\end{lemma}

\begin{prf}
Let $C_S:  \mathfrak{d} \times \mathfrak{f} \times \mathfrak{d}_1 \to \mathfrak{u}$ denote the linear map from Lemma  \ref{lem:regular_rank}. 
Let $\mathfrak{n} = \mathfrak{d} +\mathfrak{f} \subset \mathfrak{u}$. Then $\mathfrak{d} \cap \mathfrak{f} = \{ 0 \}$ implies that the restriction of $C_S$ to $ \mathfrak{d} \times \mathfrak{f}$ is an isomorphism onto $\mathfrak{n}$. This allows for the first step towards  (\ref{eqn:rank_formula})  by considering the quotient map 
\[ \overline{C}_S : \mathfrak{d}_1 \to \mathfrak{u}/\mathfrak{n}~,~ Z \mapsto S Z S^* +\mathfrak{n} ~,\]
By Lemma \ref{lem:regular_rank} and a rank formula for the quotient map we obtain that
\begin{equation} \label{eqn:rank_formula_prelim}
  {\rm rank}(d \Phi(I,S,I)) = {\rm rank}(C_S) = {\rm rank}(C|_{ \mathfrak{d} \times \mathfrak{f}})+ {\rm rank}(\overline{C}_S)= (n-1)^2 + n + {\rm rank}(\overline{C}_S)~.
\end{equation}

In order to prove that $ {\rm rank}(\overline{C}_S)= {\rm rank}({\rm Im}(S))$, we need to free ourselves from the quotient space $\mathfrak{u}/\mathfrak{n}$. For this purpose, let us consider the map 
$T: \mathfrak{u} \to \mathbb{C}^n$, $T(X)  = X \mathbf{1}$.  We claim that
\[
 T(\mathfrak{u}) = V := \{ x \in \mathbb{C}^n ~:~ {\rm Re}\langle x, \mathbf{1}\rangle = 0 \}~.
\]
Indeed, the kernel of $T$ is precisely $\mathfrak{f}$, which has real codimension $2n-1 = {\rm dim}_{\mathbb{R}}(V)$. Thus it suffices to prove that $T(\mathfrak{u}) \subset V$. This, however, follows from the fact that $X = - X^*$, for all $X \in \mathfrak{u}$, since
\begin{eqnarray*} {\rm Re} \langle T(X), \mathbf{1} \rangle & = &  {\rm Re} \langle X \mathbf{1},  \mathbf{1} \rangle = {\rm Re} \langle \mathbf{1}, X^* \mathbf{1} \rangle = 
  {\rm Re} \overline{\langle -X \mathbf{1}, \mathbf{1} \rangle} = -  {\rm Re} \langle T(X), \mathbf{1} \rangle ~.
\end{eqnarray*}

We next consider the linear map  $J: \mathfrak{u} \mapsto \br^n$, $J (X)=  {\rm Re}(T(X))$. 
The kernel of $J$ is then given by $\mathfrak{n}$, as can be seen by yet another dimension argument.
The observation that $T(\mathfrak{u}) = V$ yields that 
\[ J(\mathfrak{u}) = \{ {\rm Re} (v) : v \in T(\mathfrak{u}) \} = \{ w \in \mathbb{R}^n : \langle w, \mathbf{1} \rangle = 0 \}, \] so $J(\mathfrak{u})$ is a subspace of $\br^n$ of dimension $n-1$.  Since the real codimension of $\mathfrak{n}$ is also $n-1$, it is enough to show that 
 $\mathfrak{n} \subset {\rm ker}(J)$. For the latter, note that 
$J(X+Y) = J(X) = {\rm Re}( x_1,\ldots,x_n)$, for all $Y \in \mathfrak{f}$ and $X = {\rm diag}(x_1,\ldots,x_n) \in \mathfrak{d}$. Recall that $X \in \mathfrak{d}$ entails in addition $x_k \in i \mathbb{R}$, whence $J(X+Y)=0$, establishing that ${\rm ker}(J) = \mathfrak{n}$.

This allows to conclude that $J$ induces an {\em injective} linear map $\overline{J} : \mathfrak{u}/\mathfrak{n} \to \mathbb{R}^n$.
Hence, if we define $Q = \overline{J} \circ \overline{C}_S$, then we obtain that
${\rm rank}(\overline{C}_S) = {\rm rank}(Q)$.

Clearly, $Q:\mathfrak{d}_1\to \br^n$ and this map is much easier to treat than $\overline{C}_S$. Indeed, 
given any $Z \in \mathfrak{d}_1$, we employ $S^* \mathbf{1} = \mathbf{1}$ to compute that
\[
 Q(Z) = {\rm Re}(S Z S^* \mathbf{1}) = {\rm Re}(S Z \mathbf{1}) = - {\rm Im}(S)  z~,
\] where $z = (0,z_2,\ldots,z_n) \in \mathbb{R}^n$ is related to $Z$ via $Z = {\rm diag}(0,iz_2,\ldots,iz_n)$. 
This shows that ${\rm rank}(Q) = {\rm rank} ({\rm Im}(S)|_{W_0})$, where $W_0 \subset \mathbb{R}^n$ is the subspace of vectors with first entry equal to zero. Moreover, $W_0+ \mathbb{R} \cdot \mathbf{1}=\br^n$ and ${\rm Im}(S)\mathbf{1}=0$, since $S\mathbf{1}=\mathbf{1}$. Hence $ {\rm rank} ({\rm Im}(S)|_{W_0})= {\rm rank} ({\rm Im}(S))$.

Thus, we have established that ${\rm rank}(\overline{C_S})  = {\rm rank}({\rm Im}(S)) $, so equation (\ref{eqn:rank_formula}) follows from (\ref{eqn:rank_formula_prelim}).

\end{prf}

The above lemma leads to the fundamental fact, that for all $n\in\bn$  the maximal rank of $d\Phi$ equals the dimension of $U(n)$. 
\begin{lemma}\label{lem isi}  For all $n\in\bn$ 
 there is an $S\in {\rm Fix}(\mathbf{1}) \subset U(n)$ with ${\rm rank}(d \Phi(I,S,I)) = n^2$. 
\end{lemma}

\begin{prf}
 Let $W_\br \subset \mathbb{R}^n$ denote the orthogonal complement of $\mathbb{R} \cdot \mathbf{1}$ in $\br^n$, and $W_{\mathbb{C}} \subset \mathbb{C}^n$ denote the orthogonal complement of $\mathbb{C} \cdot \mathbf{1}$ in $\bc^n$. In order to exhibit $S\in{\rm Fix}\one$ with ${\rm rank}(d \Phi(I,S,I)) = n^2$ let us pick $S' \in {\rm Fix}(\mathbf{1}) \cap  O(n)$, where $O(n)$ stands for the orthogonal group.  This entails $S' (W_\br) = W_\br$, so we can define a unitary mapping $S$ on $W_\bc\oplus \bc\cdot 1$ by
\[
 S (v+ \alpha \mathbf{1}) = i S' v+ \alpha \mathbf{1}~,~\alpha \in \mathbb{C}, v \in W_{\mathbb{C}} ~.  
\] Then $S \in {\rm Fix}(\mathbf{1})$, with ${\rm Im}(S)|_{W_\br} = S'|_{W_\br}$. Since $S'(W_\br)= W_\br$, the restriction $ S'|_{W_\br}$ has rank $n-1$.  Moreover, $W_\br\oplus \mathbb{R} \cdot \mathbf{1}=\br^n$ and ${\rm Im}(S)\mathbf{1}=0$, so 
\[
 {\rm rank} ({\rm Im}(S))= {\rm rank} ({\rm Im}(S)|_{W_\br})={\rm rank} (S'|_{W_\br})=n-1.
\]
Thus, Lemma \ref{lem:rank_diff_simple} implies that ${\rm rank}(d \Phi(I,S,I)) = n^2$. 
\end{prf}

With the above lemma we can draw the mentioned conclusion regarding the structure of $\mathcal{B}_n$. 

\begin{theorem} For every $n \in\bn$ the set  $\mathcal{B}_n\subset U(n)$ has a nonempty open interior.
\end{theorem}
\begin{prf}
By Proposition~\ref{prop isi}  $\mathcal{B}_n =\Phi( \mathcal{D} \times {\rm Fix}(\mathbf{1}) \times \mathcal{D}_1 )$ and $\Phi$ is a map between groups of dimension $n^2$. Lemma~\ref{lem isi} guarantees an existence of a point in $ \mathcal{D} \times {\rm Fix}(\mathbf{1}) \times \mathcal{D}_1$, where the Jacobian of $\Phi$ has the full rank. Therefore, by the inverse function theorem, $\Phi$ is a diffeomorphism around this point.
\end{prf}

\begin{corollary}$\mathcal{B}_n = U(n)$ if and only if  $\mathcal{B}_n \cdot \mathcal{B}_n \subset \mathcal{B}_n$.
\end{corollary}

\begin{prf}
Clearly, $A\in\mathcal{B}_n$ iff $A^{-1}\in\mathcal{B}_n$. Therefore, $\mathcal{B}_n$ is a symmetric subset of $U(n)$ that, by the above theorem,  has a nonempty open interior.  If, in addition, it is closed under multiplication, then $\mathcal{B}_n$ must be an open subgroup. Since $U(n)$ is connected, it follows that $\mathcal{B}_n = U(n)$. The converse is clear.
\end{prf}

The results of this subsection can be compared with a theorem of De Vos and  De Baerdemacker based on Hurwitz parametrization and stating that $(\mathcal{B}_n)^{n-1}=U(n)$.

\subsection{The structure of  $\bi(A)$}

Regarding the structure of  $\bi(A)$ it is crucial to recall Conjecture~\ref{bs}. In light of the hypothesis imposed by Bj\"orck and Saffari, the main issue is to establish the cardinality of  $\bi(A)$. Therefore, there is a merit in proving the following.
\begin{proposition}
For almost all $A \in U(n)$ the set $\bi(A)$ is finite. 
\end{proposition}
\begin{prf} This follows from  Proposition~\ref{prop isi2} and Sard's theorem. By Proposition~\ref{prop isi2} we can replace $\bi(A)$ by $\Phi^{-1}(A)$. Clearly, we can concentrate on the case $\Phi^{-1}(A)\not=\emptyset$.  Recall that the regular points of a smooth map are those for which the rank of the Jacobian is full; any point that is not regular is called critical. A regular value is a value, such that all the points in the preimage are regular.   If $A$ is a regular value of $\Phi$, then $\Phi^{-1}(A)$  must be a discrete subset of the compact space $\mathcal{D} \times {\rm Fix}(\mathbf{1}) \times \mathcal{D}_1$, hence finite. Thus, if $\Phi^{-1}(A)$ is infinite, then $A$ must be the image of a critical point of $\Phi$. However,  Sard's theorem states that the images of the critical points constitute a set of the Haar measure zero, what concludes the proof.

\end{prf}

In order to achieve a deeper understanding of the set $\bi(A)$ for a given $A\in U(n)$, we consider the   {\em phasing manifold} of $A$
$$\mathcal{M}_A = \mathcal{D} A \mathcal{D}_1$$
introduced in \cite{TaZy}, and the stabilizer group  $\mathcal{D}_A=\{(E_1,E_2)\in \mathcal{D}  \times \mathcal{D}_1 : E_1AE_2=A\}$.

 By Proposition~\ref{prop isi2} the next result establishes a bijection 
\[ \bi(A) \longleftrightarrow \left( \mathcal{M}_A \cap {\rm Fix}(\mathbf{1}) \right) \times \mathcal{D}_A ~,\]
that can be explained in the following way.
Intuitively, finding all biunimodular vectors for $A$ amounts to finding all $S\in   {\rm Fix}(\mathbf{1})$ for which there exists a pair $(D_1,D_2)\in \mathcal{D}  \times \mathcal{D}_1$ such that $A=D_1SD_2$ (i.e. $S=D_1^{-1}AD_2^{-1}\in  \mathcal{M}_A $) and, after that, finding other such pairs for $S$ by considering the set $\{(D_1E_1,D_2E_2): (E_1,E_2)\in \mathcal{D}_A\}$. We shall prove that this simple scheme yields all members of $\bi(A)$ in a one-to one fashion.
\begin{theorem}\label{bije} There is a bijection $\Psi :  \left( \mathcal{M}_A \cap {\rm Fix}(\mathbf{1}) \right) \times \mathcal{D}_A \to \Phi^{-1}(A)$.
\end{theorem}
\begin{prf}
We can pick a map $\psi : \mathcal{M}_A \cap {\rm Fix}(\mathbf{1}) \to \mathcal{D}  \times \mathcal{D}_1$, $\psi(S)=(\psi_1(S),\psi_2(S))$, such that 
$\psi_1(S)\,S\,\psi_2(S)=A$. Indeed, by definition, for every $S\in \mathcal{M}_A \cap {\rm Fix}(\mathbf{1})$ there is $(D_1,D_2)\in  \mathcal{D}  \times \mathcal{D}_1$, such that $S=D_1AD_2$, thus we can pick one of such pairs and set $(\psi_1(S),\psi_2(S))=(D_1^{-1},D_2^{-1})$.

Then, the map $\Psi: \left( \mathcal{M}_A \cap {\rm Fix}(\mathbf{1}) \right) \times \mathcal{D}_A \to \Phi^{-1}(A)$ is given by
\[
\Psi\left(S,(E_1,E_2)\right)=\left(E_1\psi_1(S),S,E_2\psi_2(S)\right).
\]
To see that the image of $\Psi$ is contained in $\Phi^{-1}(A)$, it is enough to note that
\[
\Phi(\Psi\left(S,(E_1,E_2)\right))=E_1\psi_1(S)\,S\,\psi_2(S)E_2=E_1AE_2=A.
\]
The map is one-to-one, since $\Psi\left(S,(E_1,E_2)\right)=\Psi\left(S',(E'_1,E'_2)\right)$ implies that $S=S'$ and this, in turn, implies that $E_i=E_i'$, for $i=1,2$. To check that the map is onto, we take an arbitrary element $(G_1,S,G_2)\in\Phi^{-1}(A)$ and easily verify that $\Psi\left(S,(E_1,E_2)\right)=(G_1,S,G_2)$ with $E_i=G_i\,(\psi_i(S))^{-1}$, for $i=1,2$. Since
$G_1(\psi_1(S))^{-1}A\,(\psi_2(S))^{-1}G_2=G_1SG_2=A$, we get that
 $(E_1,E_2)\in\mathcal{D}_A $ and end the proof.
\end{prf}
\begin{remark}
Since in the above theorem it is hard to guarantee the continuity of $\psi$, we can not conclude that the bijection $\Psi$ is continuous. We mention this, because otherwise, by Proposition~\ref{prop isi2} the set $\bi(A)$ would be homeomorphic to  $ (\mathcal{M}_A \cap {\rm Fix}(\mathbf{1}) )\times \mathcal{D}_A$.
\end{remark}

With Theorem~\ref{bije}, we can concentrate on the set $ \left( \mathcal{M}_A \cap {\rm Fix}(\mathbf{1}) \right)\times \mathcal{D}_A $, in order to investigate the cardinality of $\bi(A)$. For that, we need a better understanding of the phasing manifold  $\mathcal{M}_A$,  for a given $A\in U(n)$. This can be obtained by 
 defining the mapping
\[
 \Phi_A : \mathcal{D} \times \mathcal{D}_1 \to U(n)~,~\Phi_A(D_1,D_2) = D_1 A D_2~
\] 
and noting that $\mathcal{M}_A = \Phi_A(\mathcal{D} \times \mathcal{D}_1)$.
This setup allows for making some useful observations.
\begin{lemma}\label{mani} For every $A \in U(n)$ the set 
$\mathcal{M}_A$ is the orbit of $A$ under the left action  of $\mathcal{D} \times \mathcal{D}_1$ on $U(n)$ given by $\left( (D_1,D_2), A \right) \mapsto D_1 A D_2$. In particular, $\mathcal{M}_A$ is a closed submanifold of $U(n)$ with $\mathcal{M}_A \simeq \left( \mathcal{D} \times \mathcal{D}_1\right) / \mathcal{D}_A$, so ${\rm dim}(\mathcal{M}_A)+{\rm dim}(\mathcal{D}_A)=2n-1$.
\end{lemma}
\begin{prf}
 Note that the map  $\left( (D_1,D_2), A \right) \mapsto D_1 A D_2$ indeed defines an action of the product group, since the diagonal group is commutative; and clearly $\mathcal{M}_A$ is the orbit of $A$. The remaining statements follow from this: $\Phi_A$ is just the quotient map with respect to this action, what implies that the differential of $\Phi_A$ has constant rank; and this in turn allows to define a manifold structure on $\mathcal{M}_A$. In addition, $\mathcal{M}_A$ is compact, hence closed, and $\Phi_A$ induces a continuous bijection $  \left( \mathcal{D} \times \mathcal{D}_1\right) / \mathcal{D}_A \to \mathcal{M}_A$ between compact Hausdorff spaces, which then has to be a homeomorphism. Therefore,  ${\rm dim}(\mathcal{M}_A)+{\rm dim}(\mathcal{D}_A)={\rm dim}( \mathcal{D} \times \mathcal{D}_1)=2n-1$.
\end{prf}

\begin{corollary}
 ${\rm dim}(\mathcal{M}_A)=2n-1 \iff  {\rm dim}(\mathcal{D}_A)=0\iff \mathcal{D}_A {\rm\,\,\, is\,\, finite.}$
\end{corollary}
\begin{proof}The first equivalence follows immediately from the dimension formula given in the above lemma. In the second equivalence one implication is obvious. For the other implication we note that, if  ${\rm dim}(\mathcal{D}_A)=0$ then $ \mathcal{D}_A$ must be a discrete subgroup of the compact group $ \mathcal{D} \times \mathcal{D}_1$, so it must be finite.
\end{proof}
 
At this stage we can provide the following explanation on the cardinality $|\bi(A)|$ of $\bi(A)$.

Obviously, for $n=1$ we have $|\bi(A)|=1$. When $n=2$, then $|\bi(A)|$ equals to 2 or to the continuum $\mathfrak{c}$. For $n=3$ we have only found examples 
 where $|\bi(A)|$ equals  to 4, 5, 6 or $\mathfrak{c}$. By Theorem~\ref{bije},  $|\bi(A)|$ is equal to the cardinality of the set $ \left( \mathcal{M}_A \cap {\rm Fix}(\mathbf{1}) \right)\times \mathcal{D}_A $. Therefore,  to see the general picture, we employ the above corollary together with Theorem~\ref{IW} and get the following:

\begin{enumerate}
 \item If ${\rm dim}(\mathcal{M}_A)<2n-1$ then
 $|\bi(A)|=\mathfrak{c}$.

 \item If ${\rm dim}(\mathcal{M}_A)=2n-1$ then
 \begin{enumerate}
\item $|\bi(A)|\in \bn$ or
\item $\bi(A)$ is infinite but countable (no examples) or
\item $\bi(A)$ is  uncountable.
\end{enumerate}
\end{enumerate}
\begin{remark} Regarding the case ${\rm dim}(\mathcal{M}_A)<2n-1$, it is clear that for the identity matrix $I$ (and any diagonal matrix) we have that
$\mathcal{M}_I=\mathcal{D}$, so ${\rm dim}(\mathcal{M}_I)=n$. Moreover, there are examples of matrices $A\in U(n)$ with ${\rm dim}(\mathcal{M}_A)$ in the range from $n$ to $2n-1$.
\end{remark}
We would like to end this section by getting closer to understanding the classical Fourier case. As it turns out, for the Fourier matrix $F_n$ the stabilizer $\mathcal{D}_{F_n}$ is trivial.

\begin{lemma} For every $n\in \bn$ we have that $\mathcal{D}_{F_n}=(I,I)$.
\end{lemma}
\begin{prf}
Let $(E_1,E_2)\in\mathcal{D}_{F_n}$, i.e. $ (E_1,E_2)\in  \mathcal{D} \times \mathcal{D}_1$ and $E_1F_nE_2=F_n$ . Then $E_2=F_n^*E_1^*F_n$. This means that $E_2$ commutes with all translations (time shifts) and all modulations (frequency shifts). A result going back to Schwinger asserts that translations and modulations generate a basis for the space of all matrices $\mathfrak{gl}(n,\bc)$.\footnote{If $T$ is a time shift by one, and $M$ is a similar frequency shift, then $\{\frac1{\sqrt n}T^kM^l:k,l=1,\dots,n\}$ form an orthonormal basis of $\mathfrak{gl}(n,\bc)$ (see \cite{sch}).} Therefore, $E_2$ commutes with all matrices in this space, so $E_2=\pm I$.  Since $E_2\in\mathcal{D}_1$, we get that $E_2=I$, so $E_1=I$ as well.
\end{prf}

The above lemma together with Theorem~\ref{bije} and Lemma~\ref{mani} entail the following information on the set of all biunimodular vectors in the classical Fourier case.
\begin{theorem}
For every $n\in\bn$ there is a bijection between $\bi(F_n)$ and $ \mathcal{M}_{F_n} \cap {\rm Fix}(\mathbf{1})$, where $\mathcal{M}_{F_n}$ is the phasing manifold of dimension $2n-1$.
\end{theorem}

The above theorem indicates the resilience of Conjecture~\ref{bs}. Another approach towards this conjecture has been designed by Gabidulin. Unfortunately, even in the prime case, where some details are provided, he does not explain how to exclude that $\bi(F_n)$ is infinite but countable.
Therefore, the only accessible confirmation of the conjecture in the prime case is \cite{Ha}.

\section{Numerical Results}\label{num:tra}

\subsection{A simple algorithm and its properties}\label{sect:algo}

In the proof of Theorem~\ref{char}(b) we have shown that  $\|A\|_{\infty\to1}\le n$  for every $A\in\ U(n)$.
As we pointed out in Remark~\ref{remchar}, biunimodular vectors for $A$ are precisely those, where the norm $\|A\|_{\infty\to1}=n$ is attained. Thus, if $v$ is a biunimodular vector of $A$, then we can think of it as of a point on $\bt^n$, where the norm $\|Av\|_1$ is maximized.

In order to maximize the $\ell^1$-norm it is convenient to notice, that for every $v\in\bc^n$ with $|v|\le 1$ (entrywise) and $A\in U(n)$ we have that
\begin{equation}\label{si1}
\|Av\|_1=\<Av,\si(Av)\>=\<v,A^*(\si(Av))\>\le\|A^*(\si(Av))\|_1,
\end{equation}
where ``$\si$'' is the standard signum function defined on the support of $v\in \bc^n$ by
\[
\si(v)=\frac v{|v|},
\]
so that $v=\si(v)|v|$. Since (\ref{si1}) holds for all $v$ with $|v|\le 1$, we can safely define $\si(v)$ as zero on the complement of the support of $v$.  

Applying  (\ref{si1}) twice yields
\begin{equation}\label{si2}
\|Av\|_1\le\|A^*(\si(Av))\|_1\le\|A(\si(A^*(\si(Av))))\|_1:=\|A v'\|_1,
\end{equation}
with $v'=\si(A^*(\si(Av)))$. This leads to the following algorithm that with each step increases the $\ell^1$-norm of $A$ on consecutive vectors, and therefore, allows to search for the biunimodular vectors of $A$.
\vskip.2cm
{\bf 1. Initialization:} Take any vector $V_0\in\bt^n$.
\vskip.1cm
{\bf 2. Step:} $V_{j+1}=\si(A^*(\si(AV_j)))$ for $j=0,1,\dots$.
\vskip.2cm
\noindent
Clearly, (\ref{si2}) assures that 
\begin{equation}\label{hill}
\|AV_{j}\|_{1}\le\|AV_{j+1}\|_{1} \quad{\rm for} \quad j=0,1,\dots
\end{equation}
giving a way to stop the algorithm after reaching $J$ such that $\|AV_{J}\|_1=n$. Of course, for numerical algorithms we can only expect this to hold within a certain, predefined precision. Since we know that $\|AV_{j}\|_1\le n$ for all $j\in\bn$, we can set this precision to $\delta>0$ and stop the algorithm after reaching a $J$ such that
\begin{equation}\label{delta0}
\|AV_J\|_1>n-\delta.
\end{equation}

In this way we rediscover a phase retrieval method that has a pretty long history. To make it short, following Jaming \cite{Ja} we recall that as early as 1932, Pauli was asking whether the information on the moduli of a wave function and its Fourier transform allows to recover the phase of the wave function (i.e. Having $|\psi|$ and $|{\cal F} \psi|$ can we recover $\psi$?).  In practice, the phase retrieval problem is a serious obstacle for analysing complex valued signals, since only partial information is given on $|\psi|$ (see \cite{Can1, Can2} by Cand\`es {\it et al.} or \cite{Bau} and \cite{Fog} for recent developments on this notorious problem). Nevertheless, under original Pauli constrains, a method for recovering the phase of a signal has been provided in 1972  by Gerchberg and Saxton (see \cite{Ger}). 
Gerchberg-Saxton algorithm has several variants, and the one that is relevant to us takes the following form. Let $|\psi|=g$ and $|{\cal F}\psi|=h$. In order to recover $\psi$ define 
\[
P_{gh} f=g \,\si_1({\cal F}^* (h\,\si_1({\cal F} f)))),
\] 
where
\[
\si_1(f)=\begin{cases} \frac f{|f|} & {\rm on } \,\, \supp(f)\\
1  & {\rm off } \,\, \supp(f)

\end{cases}
\]
and starting with an $f_0$ such that $|f_0|=g$, run $f_{j}=P_{gh}^{j}(f_0)$ until it converges to a solution $\psi'$.\footnote{Unfortunately, $\psi'$ is not unique (it depends on the choice of $f_0$), so it is only a candidate for $\psi$ (see Subsection~\ref{secFou}).} This algorithm can be viewed as an alternating projection method for finding an intersection of two subsets of a Hilbert space, that was discovered by von Neumann in 1933 (\cite{vN}), in the case when the subsets are closed subspaces. This view is justified by writing $P_{gh}$ as $P_{G}\circ P_{H}$, where $P_Gf=g\,\si_1(f)$ and 
$P_Hf={\cal F}^* (h\,\si_1({\cal F} f)))$ are projections on certain subsets of a Hilbert space. 

The method of alternating projections is a topic by itself (see \cite{APM}), that is well covered in the case of convex sets (\cite{BauBor}) and compact sets with a non-empty intersection (\cite{Com}). In our setting we are looking for an intersection of two compact non-convex sets, namely $\bt^n$ and $A^*(\bt^n)$, with a guarantee that $\bt^n \cap A^*(\bt^n)\not=\emptyset$ provided by Theorem~\ref{IW}. 

Clearly, $\si_1:\bc^n\to\bt^n$ is a projection on the torus (in the sense that $\|v-\si_1(v)\|_2={\rm dist}(v,\bt^n)$), and $A^*\circ \si_1\circ A$ is a projection on $A^*(\bt^n)$. Thus, a successive application of $P'_A(v)= \si_1(A^*(\si_1(Av)))$ yields the standard alternating projection method.

 Although we use 
\[
P_A(v)= \si(A^*(\si(Av))), \quad v\in\bc^n, \,A\in U(n)
\]
instead of $P'_A$, we shall quickly confine our search to the range where both these operators agree. Thus, we shall refer to this method as ``alternating projection algorithm'' anyway.  

After providing this backdrop, we can take on understanding the workings of the algorithm. Recall that we start with an arbitrary vector $V_0\in \bt^n$ and successively apply $P_A$, by setting $V_j=P_A^jV_0$, until we reach a $J$ such that (\ref{delta0}) holds
with a certain precision $\delta>0$.

The price for the simplicity of this algorithm is its randomness, that has to be properly interpreted.\footnote{ Essentially, one has to run the algorithm with multiple starting vectors in $\bt^n$ to get satisfactory results (see Subsection~\ref{secHaar}). }

Since $v\in \bt^n$ is biunimodular for $A$ iff $\|Av\|_1=n$, we would be satisfied with finding $V_J$ such that  $\delta$ in  (\ref{delta0}) is close to zero. However,  there is no guarantee that for a given starting vector $V_0$ we will find $J\in\bn$ so that $\delta$ is small enough.  This is already apparent in the case $A=F_2=\frac1{\sqrt2}\begin{bmatrix}1 & 1\\1& -1\end{bmatrix}$ with $V_0=(1,1)$. Indeed, here $V_j=V_0$ for all $j\in\bn$, so $\|AV_j\|_1=\sqrt2\ll2$.  Therefore, in general, the algorithm has to be applied to multiple random starting vectors in $\bt^n$.

Another aspect of the randomness is due to the discontinuity of the signum function. Clearly, neither $P_A$ nor $P'_A$ is continuous. In theory, the vector 
$\mathbf{1}=(1,1,\dots,1)\in\bc^n$ is a fix point of  both $P_A$ and $P'_A$, when $A=F_n$. In practice, both algorithms may converge to more stable solutions when applied to the starting vector $\mathbf{1}$ in the Fourier case. Of course, similar problems may occur in the general case as well. In order to avoid the discontinuity, we observe that for every $A\in U(n)$ and  $v\in\bt^n$ such that $\|A v\|_1>n-\delta$ we have the following
\[
\|\mathbf{1}-|Av|\|_2^2=n-2\|Av\|_1+\|Av\|_2^2= 2(n-\|Av\|_1)<2\delta,
\]
so $|Av|\ge \frac12$ for $\delta\le\frac18$. Moreover, for such $\delta$ we get
\[
\|A^*(\si(Av))-v\|_2^2=\|\si(Av)-Av\|_2^2=\|\mathbf{1}-|Av|\|_2^2<{2\delta},
\]
so  $|A^*(\si(Av))|\ge \frac12$ as well. This allows to draw the following conclusions. 

If the stopping condition   $(\ref{delta0})$ is satisfied, then $\|\mathbf{1}-|AV_J|\|_2<\sqrt{2\delta}$. If the initial condition
\begin{equation}\label{delta1}
\|AV_0\|_1>n-\mbox{\small$\frac1{8}$}
\end{equation}
is satisfied, then for all $j\in\bn\cup\{0\}$ we have that

\begin{itemize}
\item  $V_j:=P_A^jV_0=(P'_A)^jV_0$,
 \item  $V_j\in\bt^n$, $|AV_j|\ge\frac12$ and $|A^*(\si(AV_j))|\ge\frac12$,
\end{itemize}
due to an easy induction argument. Moreover, assuming  (\ref{delta1}), it can be proved that 
\begin{itemize}
\item $\|AV_{j+1}\|_1 =\|AV_j\|_1$ iff $V_{j+1}=V_j$,
\item $\lim_{j\to\infty}\|V_{j+1}-V_{j}\|_\infty=0$.
\end{itemize}
Where the former follows from  (\ref{si1}) and  (\ref{si2}), while the proof of the latter is a bit lenghtly, so we include it in Appendix~B  containing a detailed study of the algorithm.

 Since the sequence $(\|AV_{j}\|_1)_{j\in\bn}$ is monotonic and bounded, it converges to a limit $L\le n$. In Appendix~B we show that under the initial condition (\ref{delta1}) the cluster points of $(V_{j})_{j\in\bn}$ belong to the set $S_A$ of fix points of $P_A$. This allows to assure the convergence of $V_j$, in the case when  $S_A$ is finite. Since biunimodular vectors of $A$ belong to $S_A$, we know that for some $A$ the set $S_A$ is uncountable. Thus, it it possible, that in some cases $V_{j}$ does not converge. 

This brings us to the main point of the analysis of the algorithm. If  the stopping condition  (\ref{delta0}) holds with $\delta\le\frac18$, then the initial condition (\ref{delta1}) holds with $V_0$ replaced by $V_J$. Thus, the stopping condition  (\ref{delta0}), with $\delta$ numerically near to zero, gives $V_J$ that is close to a fix point of $P_A$. If $L:=\lim_{j\to\infty}\|AV_{j}\|_1=n$, then $V_J$ is close to a biunimodular vector of $A$, however, it is not clear how close.  If $L<n$, then $V_J$ is close (in a similar fashion) to a fix point of $P_A$, that may be far away from biunimodular vectors. Moreover, it is impossible to establish numerically that $L=n$. Hence, we need to settle for {\it  $\delta$-near biunimodular} vectors given as
\[
 \bi_\delta(A) = \{ v \in \mathbb{T}^n~:~\| A v \|_1 > n-\delta \}
\] and stress again, that some of these vectors can be far away from biunimodular vectors, even if $\delta$ is nearly zero.

 In short, if the algorithm returns vectors that are ``biunimodular'' in practice, there is no guarantee that they are close to the theoretical biunimodular vectors.

Despite this uncertainty, it turns out, that finding  near biunimodular vectors for every unitary matrix is sufficient to conclude, that all such matrices have biunimodular vectors. In the following proposition we prove even more.
\begin{proposition}\label{prenal} Let $\al<1$. If for every $n\in\bn$ and for every $A\in U(n)$ there is $v\in\bt^n$ such that
\begin{equation}\label{enal}
\| A v \|_1 \ge n-n^\al,
\end{equation}
then every unitary matrix has a biunimodular vector.
\end{proposition}
\begin{prf} This follows from considering diagonal block matrices. If $B\in U(m)$ does not have biunimodular vectors, then    $\|B\|_{\infty\to1}=m-\e$ with some $\e>0$. Therefore, for every $N\in\bn$, the block matrix $A_N$ with $N$ copies of $B$ on the diagonal
\[
A_N:=\left[ \: 
\begin{array}{*{4}{c}}
\cline{1-1}
\multicolumn{1}{|c|}{B}&  &&\\
\cline{1-2}
& \multicolumn{1}{|c|}{B} &   &  
\\
\cline{2-2}
& &\ddots &
 \\
\cline{4-4}
& &  & \multicolumn{1}{|c|}{B}
  \\
\cline{4-4}
\end{array}
\,\,\right]
\]
belonging to $U(n)$ with $n=Nm$, has the norm   $\|A_N\|_{\infty\to1}=N(m-\e)$. Thus, if (\ref{enal}) holds for $A_N$, then
we get that 
\[
n-n^\al\le\|A_Nv\|_1\le n-n\mbox{\small$\frac\e{m}$},
\]
leading to a contradiction $\frac\e m\le n^{\al-1}$, since $n^{\al-1}\to 0$ as $N\to\infty$.
\end{prf}

In the next subsection we check the performance of the alternating projection algorithm for unitary matrices up to dimension 100. The final subsection contains results of the application of this algorithm to the Fourier case.

\subsection{Effectiveness of the algorithm}\label{secHaar}
 
In order to measure the performance of the alternating projection algorithm we
recall, that for $\delta>0$ and $A \in U(n)$ we have defined the set
\[
 \bi_\delta(A) = \{ v \in \mathbb{T}^n~:~\| A v \|_1 > n-\delta \}
\] of $\delta$-near biunimodular vectors of $A$. For a given unitary matrix $A$ and a prescribed parameter $\delta$ the algorithm is supposed to return a member of $ \bi_\delta(A)$.
The effectiveness of the algorithm can be measured by considering the set $\mathcal{B}_{n,\delta}$, that is, the set of these matrices in $U(n)$, for which the algorithm returns the desired outcome.

The aim of this subsection is to provide numerical evidence that the Haar measure of $\mathcal{B}_{n,\delta}$ (that we denote by  $\mu(\mathcal{B}_{n,\delta})$)  is very close to one. 
For this purpose we proceeded as follows:  For each dimension under consideration, we drew a fixed number $K$ of random unitary matrices with Haar measure used as underlying probability measure. To achieve this, we implemented the method described in \cite{Mezzadri}. We then ran the alternating projection algorithm on the matrix with randomly chosen starting vectors, in order to produce near biunimodular vectors. The search was terminated either if one such vector was found, or until a maximal number of starting points was exceeded. 
The parameters used in this procedure were as follows:
\begin{itemize}
\item $\delta = 10^{-10}$. 
 \item Dimensions: $n=3,5,10,25,50,75,100$.
\item Number of matrices: $K = 10^4$;
\item Maximal number of starting points: $M=10^3$;
\item Maximal number of iterations: $L = 10^4$.
\end{itemize}

We implemented the described algorithms in MATLAB, and ran them on a stand-alone personal computer\footnote{The CPU of the machine was a 6 core AMD FX 6100 processor with 8 MB level 1 cache, running at 3.3. GHz, together with 8 GB RAM. The operating system was SUSE linux 12.2 for 64-bit PCs.}. 

The results of our numerical experiment are documented in Table \ref{tab:results_numsearch}. Two facts seem particularly striking to us: First of all, the algorithm found a $\delta$-near biunimodular vector {\em for each of the $10^4$ random matrices in each dimension considered}. This translates to an estimate of $\mu(\mathcal{B}_{n,\delta})$, which holds at least with very high likelihood: Our numerical experiment amounts to performing a Bernoulli experiment based on repeated independent samples of a Haar-distributed random matrix $A$ in $U(n)$, and a positive outcome of the experiment (a near biunimodular vector being found) implies $A \in \mathcal{B}_{n,\delta}$.
Thus the probability of a positive outcome for a single event is given by $p \le \mu(\mathcal{B}_{n,\delta})$. Assuming that $\mu(\mathcal{B}_{n,\delta}) \le 1-10^{-3} = 0.999$, the probability of obtaining a positive result  $10^4$ times in a row will thus be $\le 0.999^{10^4} \approx 4.5173 \cdot 10^{-5}$. In other words, our experiments provide very convincing evidence that $\mu(\mathcal{B}_{n,\delta})>0.999$. 

Another striking phenomenon is illustrated by the $3$rd and $4$th column of Table \ref{tab:results_numsearch}, highlighting the effectiveness of the search algorithm. 
For small to medium dimensions and in the average case, the algorithm needs only very few starting points to find a near biunimodular vector, but even the maximal number of such points stays well away from the maximal number of $10^3$ that was allowed. As the dimension increases, several effects conspire to increase the algorithm complexity: The rate of convergence in the alternating projection algorithm slows down, as witnessed by the increasing percentage of cases where the maximal number of iterations is reached before the $(n-\delta)$-threshold is passed (not shown in the table). Also, one may expect that the chances of randomly picking starting points for which the alternating projection algorithm converges sufficiently quickly also diminish with increasing dimension. Both effects increase the number of starting points needed in the search. In addition, the costs of applying the matrix-by-vector multiplication also can be expected to contribute to a nonlinear increase of the computational load. It seems 
that the overall 
result is an exponential increase of runtime: A log-linear fit for the runtimes for $n \ge 10$ revealed a behaviour like $O(e^{0.05n})$. 

\begin{table}
\centering
 \begin{tabular}{|l|c|c|c|l|} \hline 
 Dimension & Matrices in $\mathcal{B}_{n,\delta}$ &  starting points &  starting points & total runtime \\
            & & (average) & (maximal) & (rounded) \\ \hline 
          3 & $10^4$ & $1.0663$ & $5$ &  $1$ min $38$ sec \\ \hline
          5 & $10^4$ & $1.1719$ & $7$ & $4$ min $20$ sec \\ \hline
         10 & $10^4$ & $1.3472$ & $7$ &  $9$ min $59$ sec \\ \hline
         25 & $10^4$ & $1.6521$ & $10$ & $28$ min $51$ sec \\ \hline
         50 & $10^4$ & $2.4106$ & $17$ &  $1$ h $34$ min \\ \hline
         75 & $10^4$ & $4.3361$ & $34$ &  $5$ h $6$ min \\ \hline 
        100 & $10^4$ &   $9.4730$  & $89$ & $16$ h $40$ min \\ \hline 
 \end{tabular}
\caption{Results of the numerical experiment}
\label{tab:results_numsearch}
\end{table}

\subsection{Biunimodular vectors for Fourier matrices}\label{secFou}
The search for the biunimodular vectors for $F_n$ via the alternating projection algorithm is the simplest realization of the phase retrieval process under trivial Pauli constraints $|\psi|\equiv |{\cal F}\psi|\equiv 1$ in the discrete Fourier case ${\cal F}=F_n$, $\psi\in\bc^n$.

When looking for elements of $\bi(F_n)$ for a Fourier matrix $F_n$, it is advantageous to take into account various operations that preserve this set: 
\begin{itemize}
 \item Circular shifts by $k = 0,1,\ldots,n-1$. I.e., if $(u_0,\ldots,u_{n-1}) \in \bi(F_n)$, then the vector \\
 $(u_{k},u_{k+1},\ldots,u_{n-1},u_0,\ldots,u_{k-1})$ is biunimodular. 
 \item Modulation by $k = 0,1,\ldots,n-1$: If $(u_0,\ldots,u_{n-1})\in \bi(F_n)$, then the vector  \\
 $(u_0, u_1 \exp(2 \pi i k/n),\ldots, u_{n-1} \exp(2 \pi i k(n-1)/n))$ is biunimodular.
 \item Dilation by $k =0,1,\ldots,n-1$ coprime with $n$: If $(u_0,\ldots,u_{n-1})\in \bi(F_n)$, then
 $(y_0,\ldots,y_{n-1})$ with $y_j = u_{jk \mod n}$ is biunimodular.
 \item Conjugation: If $(u_0,\ldots,u_{n-1})\in \bi(F_n)$, then its conjugate $(\overline{u}_0,\ldots,\overline{u}_{n-1})\in \bi(F_n)$. 
 \item Fourier transform: If $u \in \bi(F_n)$, then $F_n u$ is biunimodular as well.  
\end{itemize}
These operations induce the action of a finite group $G_n$ of size $4 n^2 \varphi(n)$ on $\bi(F_n)$, where $\varphi(n)$ is Euler's totient function. 

We employed the alternating projection method to produce a list of near biunimodular vectors $v \in \bi_\delta(F_n)$, for $\delta = 10^{-7}$. More precisely, we computed a list of orbit representatives of such vectors, by repeatedly running the alternating projection algorithm with random starting points. For all dimensions, we used $10^4$ starting vectors, and at most 30 000 iterations before starting with a new random vector. Whenever the algorithm terminated in a $\delta$-near biunimodular vector, the result was appended to the list, provided its $\ell^2$-distance to all orbits already on the list exceeded the threshold $\tau = 10^{-5}$.  

There is a certain degree of arbitrariness in our choices of parameters. Regarding the choice of $\tau$ and $\delta$, recall that we currently have no means of predicting how close a $\delta$-near biunimodular vector is to an actual biunimodular vector (if at all). Also, an a priori choice of the threshold $\tau$ would have to be based on the knowledge of the local maxima of the mapping $v \mapsto \| A v \|_1$ on the torus. The above choice of $\tau$ was motivated by the observation that the orbits of known biunimodular vectors in dimensions up to 13 had a minimal $\ell^2$-distance $\ge 10^{-2}$ and $10^{-5}$ is safely below this value. 

The procedure to compute near biunimodular vectors was applied to all square-free numbers between 2 and 15. In the following presentation of numerical results, we say that an orbit found by our algorithm is close to one from the literature, if the minimal distance (measured in $\| \cdot \|_2$) between them is $< 10^{-5}$. We observed that each time an orbit found by alternating projections was close to an orbit from the literature, the orbit cardinalities matched as well.

\begin{itemize}
 \item $n=2$: The algorithm precisely found the one orbit consisting of the two vectors $(1,\pm i)$.
 \item $n=3$: The algorithm found one orbit of length 6, which was close to the orbit of the gaussian vector $(1,\om^2,1)$, where $\om=\exp(2 \pi i/3)$.
 \item $n=5$: The algorithm found two orbits, each of length 10. Each orbit was close to one of the orbits represented by the gaussian vectors
$(1,\om,\om^4,\om^4,\om)$ and $(1,\om^2,\om^3,\om^3,\om^2)$, where $\om =
\exp(2 \pi i/5)$.
 \item $n=6$: The algorithm found two orbits, one of length 12 and one of length 36. The orbit of length 12 was close to the gaussian vector,
the orbit of length 36 was close to the vector described
in \cite[formula (3.20)]{Haa}.
 \item $n=7$: The algorithm found three orbits, with lengths $42,196$ and $294$, yielding 532 solutions in total. The orbit of length $42$ is
close to the gaussian solution, whereas the orbit of
length 196 is close to the Bj\"orck sequence $(1,1,1,e^{i \theta},1,e^{i \theta},e^{i \theta})$ mentioned in the introduction. The orbit of
length 294 was close to the solution given in \cite[formula
(3.24)]{Haa}.
 \item $n=10$: Here we obtained 10 orbits, with orbit lengths 20 (2 orbits), 200 (3 orbits), 400 (4 orbits), and 800 (1 orbit), resulting in a
total of 3040 solutions. For $n=10$, the literature does
not provide a complete list of solutions.
 \item $n=11$: Here 12 orbits were found, with lengths 110 (1), 484 (1), 1210 (7), 2420 (2)  and 4840 (1), resulting in a total number of
18744 vectors. Among these were the gaussians, all contained
in the orbit of length 110. For $n=11$, the literature does not provide a complete list of solutions.
 \item $n = 13$: The algorithm found 20 orbits, with lengths 78 (2), 338 (2), 1014 (2), 1352 (2), 2028 (7), and 4056 (5), yielding 40040
solutions in total. Gabidulin and Shorin \cite{GS} exhibited 25
orbits for $n=13$, with 53222 solutions overall. Each of the orbits found via alternating projections was close to precisely one orbit
representative from the list presented by Gabidulin and Shorin.
The following orbits given in \cite{GS} were not found by the algorithm: An orbit with 1014 elements, corresponding to the case $e=6$ in
\cite[Subsection 7.4]{GS}, two orbits with 2028 solutions,
corresponding to lines one and seven in the table in \cite[Subsection 7.4]{GS}, and the first two orbits listed for the case $e=12$, each
having 4056 elements.
 \item $n=14$: The algorithm found 39 orbits, with lengths 84 (1), 196 (2), 392 (1), 588 (3), 784 (2), 1176 (12), 2352 (13), 4704 (5),
resulting in a total of 72408 vectors. Among these were the
gaussians, all contained in the orbit of length 84. For $n=14$, the literature does not provide a complete list of solutions.
 \item $n=15$: The algorithm produced 46 orbits, with lengths 900 (2), 1800 (21), 3600 (20), 7200 (3), resulting in a total of 133200 vectors.
The gaussian vectors are contained in two orbits of length
60, which were not found by the algorithm.
 For $n=15$, the literature does not provide a complete list of solutions.
\end{itemize}

\section* {Appendix A}

\newcounter{bla}
\renewcommand{\thebla}{A.\arabic{bla}}
\newtheorem{sobserv}[bla]{Observation}
\newtheorem{scor}[bla]{Corollary}

As we have already explained, Theorem~\ref{bimat} allows for an easy construction of all unitary matrices in the dyadic case $U(2^n)$. Recall, that every such matrix can be obtained from four matrices in $U(2^{n-1})$. For $n=1$ the procedure yields $U(2)$ via (\ref{u2f2}). Here, we show how the procedure works in the case $n=2$, to get a closed formula for $U(4)$.

For a collection of 16 unimodular parameters $a_j,b_j,c_j,z_j\in\bt$ with $j=0,1,2,3$ we define the corresponding four matrices in U(2) via  (\ref{u2f2}):
\begin{align}\label{macierze}
A&=\frac 12\begin{bmatrix}
a_1(1+a_0) & a_1a_3(1-a_0) \\
a_2(1-a_0) & a_2a_3(1+a_0)
\end{bmatrix}
&B=\frac 12\begin{bmatrix}
b_1(1+b_0) & b_1b_3(1-b_0) \\
b_2(1-b_0) & b_2b_3(1+b_0)
\end{bmatrix}\\
C&=\frac 12\begin{bmatrix}
c_1(1+c_0) & c_1c_3(1-c_0) \\
c_2(1-c_0) & c_2c_3(1+c_0)
\end{bmatrix}
&Z=\frac 12\begin{bmatrix}
z_1(1+z_0) & z_1z_3(1-z_0) \\
z_2(1-z_0) & z_2z_3(1+z_0)
\end{bmatrix}&.\nonumber
\end{align}
Theorem~\ref{bimat} asserts that every unitary matrix $U\in U(4)$ with entries $(U)_{k,l}=u_{k,l}$, where $k,l=1,2,3,4$, can be obtained as
\[U=\frac 12\begin{bmatrix}
A(I+Z) & A(I-Z)C \\
B(I-Z) & B(I+Z)C
\end{bmatrix}.
\]
This gives the following description of $U(4)$:

$$u_{11}={\scriptstyle \frac12} ({\scriptstyle \frac14} a_1 a_3 (1 - a_0) z_2 (1 - z_0) + 
      {\scriptstyle \frac12} a_1 (1 + a_0) (1 + {\scriptstyle \frac12} z_1 (1 + z_0))) $$
$$u_{21}={\scriptstyle \frac12} ({\scriptstyle \frac14} a_2 a_3 (1 + a_0) z_2 (1 - z_0) + 
      {\scriptstyle \frac12} a_2 (1 - a_0) (1 + {\scriptstyle \frac12} z_1 (1 + z_0))) $$
$$u_{31}={\scriptstyle \frac12} (-{\scriptstyle \frac14} b_1 b_3 (1 - b_0) z_2 (1 - z_0) + 
      {\scriptstyle \frac12} b_1 (1 + b_0) (1 - {\scriptstyle \frac12} z_1 (1 + z_0))) $$
$$u_{41}={\scriptstyle \frac12} (-{\scriptstyle \frac14} b_2 b_3 (1 + b_0) z_2 (1 - z_0) + 
      {\scriptstyle \frac12} b_2 (1 - b_0) (1 - {\scriptstyle \frac12} z_1 (1 + z_0))) $$

$$u_{12}=   {\scriptstyle \frac12} ({\scriptstyle \frac14}a_1 (1 + a_0) z_1 z_3 (1 - z_0) + 
      {\scriptstyle \frac12} a_1 a_3 (1 - a_0) (1 + {\scriptstyle \frac12} z_2 z_3 (1 + z_0)))$$
$$u_{22}=
   {\scriptstyle \frac12} ({\scriptstyle \frac14}a_2 (1 - a_0) z_1 z_3 (1 - z_0) + 
      {\scriptstyle \frac12} a_2 a_3 (1 + a_0) (1 + {\scriptstyle \frac12} z_2 z_3 (1 + z_0)))
$$
$$u_{32}= {\scriptstyle \frac12} (-{\scriptstyle \frac14} b_1 (1 + b_0) z_1 z_3 (1 - z_0) + 
      {\scriptstyle \frac12} b_1 b_3 (1 - b_0) (1 - {\scriptstyle \frac12} z_2 z_3 (1 + z_0)))
$$
$$u_{42}=   {\scriptstyle \frac12}  (-{\scriptstyle \frac14} b_2 (1 - b_0) z_1 z_3 (1 - z_0) + 
      {\scriptstyle \frac12} b_2 b_3 (1 + b_0) (1 - {\scriptstyle \frac12} z_2 z_3 (1 + z_0)))
$$

$$u_{13}={\scriptstyle \frac12} ({\scriptstyle \frac12} c_1  (1 + c_0) (-{\scriptstyle \frac14}a_1 a_3 (1 - a_0) z_2 (1 - z_0) + 
      {\scriptstyle \frac12} a_1 (1 + a_0) (1 - {\scriptstyle \frac12} z_1 (1 + z_0))) +$$
$$ 
   {\scriptstyle \frac12} c_2 (1 - c_0) (-{\scriptstyle \frac14}a_1 (1 + a_0) z_1 z_3 (1 - z_0) + 
      {\scriptstyle \frac12} a_1 a_3 (1 - a_0) (1 - {\scriptstyle \frac12} z_2 z_3 (1 + z_0))))
$$
$$u_{23}={\scriptstyle \frac12} ({\scriptstyle \frac12} c_1 (1 + c_0) (-{\scriptstyle \frac14} a_2 a_3 (1 + a_0) z_2 (1 - z_0) + 
      {\scriptstyle \frac12} a_2 (1 - a_0) (1 - {\scriptstyle \frac12} z_1 (1 + z_0))) +$$
$$ 
   {\scriptstyle \frac12} c_2 (1 - c_0) (-{\scriptstyle \frac14}a_2 (1 - a_0) z_1 z_3 (1 - z_0) + 
      {\scriptstyle \frac12} a_2 a_3 (1 + a_0) (1 - {\scriptstyle \frac12} z_2 z_3 (1 + z_0))))
$$
$$u_{33}={\scriptstyle \frac12} ({\scriptstyle \frac12} c_1  (1 + c_0) ({\scriptstyle \frac14} b_1 b_3 (1 - b_0) z_2 (1 - z_0) + 
      {\scriptstyle \frac12} b_1 (1 + b_0) (1 + {\scriptstyle \frac12} z_1 (1 + z_0))) +$$
$$ 
   {\scriptstyle \frac12} c_2 (1 - c_0) ({\scriptstyle \frac14} b_1 (1 + b_0) z_1 z_3 (1 - z_0) + 
      {\scriptstyle \frac12} b_1 b_3 (1 - b_0) (1 + {\scriptstyle \frac12} z_2 z_3 (1 + z_0))))
$$
$$u_{43}={\scriptstyle \frac12} ({\scriptstyle \frac12} c_1  (1 + c_0) ({\scriptstyle \frac14} b_2 b_3 (1 + b_0) z_2 (1 - z_0) + 
      {\scriptstyle \frac12} b_2 (1 - b_0) (1 + {\scriptstyle \frac12} z_1 (1 + z_0))) +$$
$$ 
   {\scriptstyle \frac12} c_2 (1 - c_0) ({\scriptstyle \frac14} b_2 (1 - b_0) z_1 z_3 (1 - z_0) + 
      {\scriptstyle \frac12} b_2 b_3 (1 + b_0) (1 + {\scriptstyle \frac12} z_2 z_3 (1 + z_0))))
$$

$$u_{14}={\scriptstyle \frac12} ({\scriptstyle \frac12} c_1 c_3 (1 - c_0) (-{\scriptstyle \frac14}a_1 a_3 (1 - a_0) z_2 (1 - z_0) + 
      {\scriptstyle \frac12} a_1 (1 + a_0) (1 - {\scriptstyle \frac12} z_1 (1 + z_0))) +$$
$$ 
   {\scriptstyle \frac12} c_2 c_3 (1 + c_0) (-{\scriptstyle \frac14}a_1 (1 + a_0) z_1 z_3 (1 - z_0) + 
      {\scriptstyle \frac12} a_1 a_3 (1 - a_0) (1 - {\scriptstyle \frac12} z_2 z_3 (1 + z_0))))
$$
$$u_{24}={\scriptstyle \frac12} ({\scriptstyle \frac12} c_1 c_3 (1 - c_0) (-{\scriptstyle \frac14} a_2 a_3 (1 + a_0) z_2 (1 - z_0) + 
      {\scriptstyle \frac12} a_2 (1 - a_0) (1 - {\scriptstyle \frac12} z_1 (1 + z_0))) +$$
$$ 
   {\scriptstyle \frac12} c_2 c_3 (1 + c_0) (-{\scriptstyle \frac14}a_2 (1 - a_0) z_1 z_3 (1 - z_0) + 
      {\scriptstyle \frac12} a_2 a_3 (1 + a_0) (1 - {\scriptstyle \frac12} z_2 z_3 (1 + z_0))))
$$
$$u_{34}={\scriptstyle \frac12} ({\scriptstyle \frac12} c_1 c_3 (1 - c_0) ({\scriptstyle \frac14} b_1 b_3 (1 - b_0) z_2 (1 - z_0) + 
      {\scriptstyle \frac12} b_1 (1 + b_0) (1 + {\scriptstyle \frac12} z_1 (1 + z_0))) +$$
$$ 
   {\scriptstyle \frac12} c_2 c_3 (1 + c_0) ({\scriptstyle \frac14} b_1 (1 + b_0) z_1 z_3 (1 - z_0) + 
      {\scriptstyle \frac12} b_1 b_3 (1 - b_0) (1 + {\scriptstyle \frac12} z_2 z_3 (1 + z_0))))
$$
$$u_{44}={\scriptstyle \frac12} ({\scriptstyle \frac12} c_1 c_3 (1 - c_0) ({\scriptstyle \frac14} b_2 b_3 (1 + b_0) z_2 (1 - z_0) + 
      {\scriptstyle \frac12} b_2 (1 - b_0) (1 + {\scriptstyle \frac12} z_1 (1 + z_0))) +$$
$$ 
   {\scriptstyle \frac12} c_2 c_3 (1 + c_0) ({\scriptstyle \frac14} b_2 (1 - b_0) z_1 z_3 (1 - z_0) + 
      {\scriptstyle \frac12} b_2 b_3 (1 + b_0) (1 + {\scriptstyle \frac12} z_2 z_3 (1 + z_0))))
$$

\vskip.3cm
\noindent
with 16 parameters $a_j,b_j,c_j,z_j\in\bt$, where $j=0,1,2,3$.

By setting $u_{11}=1$ in the above formula for $U(4)$ we will find a description of $U(3)$, that is used  in Theorem~\ref{u3} to prove the existence of a biunimodular vector for every matrix in $U(3)$. This shall be executed as a series of observations. 

First, we set  $u_{11}=1$ in the above formula for $U(4)$.
\begin{sobserv}
$u_{11}:={\scriptstyle \frac12} ({\scriptstyle \frac14} a_1 a_3 (1 - a_0) z_2 (1 - z_0) + 
      {\scriptstyle \frac12} a_1 (1 + a_0) (1 + {\scriptstyle \frac12} z_1 (1 + z_0))) =1$ iff
$a_0=a_1=z_0=z_1=1$.
\end{sobserv}
\begin{prf}
Clearly, $u_{11}=\frac12\<R,e+\ov K\>$, where $e=(1,0)$,  $R=\frac12(a_1(1+a_0) , a_1a_3(1-a_0))$ denotes the the first row of $A$ from (\ref{macierze}) and $K=\frac12(z_1(1+z_0),
z_2(1-z_0))$ denotes the first collumn of $Z$ from therein. Thus, $\|R\|_2=\|K\|_2=1$.

 Since $u_{11}$  is an entry of a unitary matrix, we know that $|u_{11}|\le 1$. This inequality can be proved directly in the following way 
\[
|u_{11}|=\mbox{\small$\frac1{2}$}\left|\<R,e+\ov K\>\right|\le\mbox{\small$\frac1{2}$}\left(|\<R,e\>|+|\<R,\ov K\>|\right)\le \mbox{\small$\frac1{2}$}\left(\|R\|\|e\|+\|R\|\| K\|\right)=1.
\]
Therefore, $|u_{11}|=1$ means that $|\<R,e\>|=\|R\|\|e\|$, $|\<R,\ov K\>|=\|R\|\| K\|$ (so $R=\la e$, $K=\mu e$ with $\la,\mu\in\bt$) and $\left|\<R,e+\ov K\>\right|=|\<R,e\>|+|\<R,\ov K\>|$ (so $\mu=1$). Finally, $u_{11}=\frac12\<R,e+\ov K\>=1$ gives $\la=1$. Hence, $u_{11}=1$ iff $R=K=e$, that is, $a_0=a_1=z_0=z_1=1$.
\end{prf}

Then, we set $a_0=a_1=z_0=z_1=1$ in the formula for $U(4)$, to get a corresponding formula for $U(3)$ given by the simplified entries $u_{k,l}$ with $k,l=2,3,4$. By pulling out some factors from these entries we obtain the following matrix: $D_1 U_{(a,b,c,z)} D_2$, where $D_1$ is the diagonal matrix 
$D_1=\diag(a_2a_3,b_1b_3,b_2b_3)$, $D_2=\diag(1,c_2,c_3)$ and

\[
U_{(a,b,c,z)}=\begin{bmatrix}
\frac{(1+a)}2 & \frac{(1-a)(1-c)}4 &  \frac{(1-a)(1+c)}4\\
&&&\\
\frac{(1-a)(1-b)}4 & z\frac{(1+b)(1+c)}4+\frac{(1+a)(1-b)(1-c)}8 &  z\frac{(1+b)(1-c)}4+\frac{(1+a)(1-b)(1+c)}8\\
&&&\\
\frac{(1-a)(1+b)}4 & z\frac{(1-b)(1+c)}4+\frac{(1+a)(1+b)(1-c)}8 &  z\frac{(1-b)(1-c)}4+\frac{(1+a)(1+b)(1+c)}8
\end{bmatrix},
\]
with $a=z_2z_3$, $b=b_4$, $c=c_4$ and $z=c_1\ov {c_2  b_3}$. This proves the following description of $U(3)$
\begin{sobserv}\label{obsa2}
\[
U(3)=\{\diag(\la_1,\la_2,\la_3)U_{(a,b,c,z)}\diag(1,\la_4,\la_5): \, a,b,c,z,\la_j\in\bt, j=1,\dots,5\}.
\]
\end{sobserv}
By writing $a=e^{i2\al}$, $b=e^{i2\be}$, $c=e^{i2\ga}$ we get that $U(a,b,c,z)=D_1'T_{(\al,\be,\ga,z)}D_2'$,
where
$D_1'=\diag(e^{i\al},e^{i(\al+\be)},-ie^{i(\al+\be)})$, $D_2'=\diag(1,e^{i\ga},-ie^{i\ga})$ and 
\[
T_{(\al,\be,\ga,z)}=\begin{bmatrix} \cos\al &-\sin\al\sin\ga&\sin\al\cos\ga\\
-\sin\al\sin\be&z\cos\be\cos\ga-\cos\al\sin\be\sin\ga&z\cos\be\sin\ga+\cos\al\sin\be\cos\ga\\
\sin\al\cos\be&z\sin\be\cos\ga+\cos\al\cos\be\sin\ga&z\sin\be\sin\ga-\cos\al\cos\be\cos\ga
\end{bmatrix},
\]
providing a parallel description of $U(3)$:

\begin{sobserv}\label{obsa3}
\[
U(3)=\{\diag(\la_1,\la_2,\la_3)T_{(\al,\be,\ga,z)}\diag(1,\la_4,\la_5): \, \al,\be,\ga\in[0,2\pi], z,\la_j\in\bt, j=1,\dots,5\}.
\]
\end{sobserv}

\section*{Appendix B}\label{apd}
\newcounter{dla}
\renewcommand{\thedla}{B.\arabic{dla}}
\newtheorem{dproposition}[dla]{Proposition}
\newtheorem{dlemma}[dla]{Lemma}
\newtheorem{dcorollary}[dla]{Corollary}

In order to conduct a more detailed study of the alternating projection algorithm we begin by proving the following
\begin{dproposition}\label{propalg} Let $A\in U(n)$, $v\in\bt^n$ and $P_A(v)= \si(A^*(\si(Av)))$. Then, the condition $\|A v\|_1>n-\delta$ with $0<\delta\le\frac1{8}$ implies that
\begin{enumerate}
 \item[(a)]  $\|\mathbf{1}-|Av|\|_2<\sqrt{2\delta}$,
 \item[(b)] $|Av|\ge\frac12$ and  $|A^*(\si(Av))|\ge \frac12$,
\item[(c)] $\|AP_Av\|_1 =\|Av\|_1$ iff $P_Av=v$,
\item[(d)] $\|P_Av-v\|_\infty\le 2 \sqrt n (\|AP_Av\|_1 -\|Av\|_1)^\frac12$.
\end{enumerate}

\end{dproposition}

\begin{prf}

To see that $|Av|$ is  close to $\mathbf{1}=(1,1,\dots,1)$ we use $\|v\|_2=\sqrt{n}$ and observe that
\[
\|\mathbf{1}-|Av|\|_2^2=n-2\|Av\|_1+\|Av\|_2^2= 2(n-\|Av\|_1)<2\delta,
\]
proving (a). This already assures that $|Av|\ge\frac12$ for $\delta\le\frac18$ (otherwise
$\frac14\le\|\mathbf{1}-|Av|\|_2^2<2\delta$, forcing $\delta>\frac18$). 

Since  $|Av|>0$, we easily see that $\|\si(Av)-Av\|_2=\|\mathbf{1}-|Av|\|_2<\sqrt{2\delta}$. Thus,
\[
\|A^*(\si(Av))-v\|_2<\sqrt{2\delta}.
\]
Therefore, by the same token as before, $\delta\le\frac18$ implies that $|A^*(\si(Av))|\ge \frac12$.

To prove (c) we note that, by (\ref{si1}) and (\ref{si2}),
\[
\|Av\|_1=\<v,A^*(\si(Av))\>\le\|A^*(\si(Av))\|_1\le\|A(\si(A^*(\si(Av))))\|_1=\|A P_Av\|_1
\]
and, therefore, 
\begin{equation}\label{si4}
0\le \|A^*(\si(Av))\|_1-\<v,A^*(\si(Av))\>\le \|A P_Av\|_1-\|Av\|_1.
\end{equation}
Thus, $\|A P_Av\|_1=\|Av\|_1$  implies that $\|A^*(\si(Av))\|_1=\<v,A^*(\si(Av))$, what entails $v=\si(A^*(\si(Av)))=P_Av$, since
$v\in\bt^n$ and $|A^*(\si(Av))|>0$.

To show (d) we remark, that 
$|A^*(\si(Av))|\ge \mbox{\small$ \frac12$}$
 together with $\|A^*(\si(Av))\|_1\le n$ and  (\ref{si4}) allows to apply Lemma~\ref{si8} below, proving (d).
\end{prf}

To show Lemma~\ref{si8} we start with an elementary observation.
\begin{dlemma}\label{si5}
Let $a,b\in\bc$ and $0\le\e<1$. If $a+b=1-\e$ and $|a|+|b|=1$ then $||a|-a|\le\sqrt\e$.
\end{dlemma}
\begin{prf}
Since $|1-\e-a|=1-|a|$, we get that $(1-\e)2\re( a)=(1-\e)^2-1+2|a|$, so
\begin{equation}\label{si6}
||a|-a|^2=2|a|(|a|-\re(a))=\e|a|(2-\e-2\re(a))=\e|a|\left(1+\frac{1-2|a|}{1-\e}\right)\le\e,
\end{equation}
what is clear for $|a|\ge\frac12$, and follows from $\e\le2|a|$ in the case when $|a|\le \frac12$.
\end{prf}

The next lemma can be treated as a solution to a certain random walk problem on the complex plane.
\begin{dlemma}\label{si7}
Let $u\in\bc^n$. If $\<u,\mathbf{1}\>>0$, then $$\||u|-u\|^2_\infty\le \|u\|_1(\|u\|_1-\<u,\mathbf{1}\>).$$
\end{dlemma}
\begin{prf}
Let $u=(u_1,u_2,\dots,u_n)\in\bc^n\setminus\{0\}$. Due to rescaling we can assume that $\|u\|_1=1$. Moreover, the problem is insensitive to permuting the entries of $u$, so it is enough to show that
$||u_1|-u_1|^2\le 1-\<u,\mathbf{1}\> $. The only hard part is to realize, that this estimate boils down to the previous lemma.

 Let us denote $ 1-\<u,\mathbf{1}\>$ by $\e$. Since  $0<\<u,\mathbf{1}\>\le\|u\|_1=1$, it is clear that $0\le\e<1$. Moreover, let us set $a:=u_1$ and $x:=\sum_{k=2}^nu_k$. Clearly, $a+x=1-\e$. This means that the circle with the center at $a$ and radius $|x|$ intersects the interval $[0,1]$ at the point $1-\e$. Since $|x|\le\sum_{k=2}^n|u_k|= 1- |a|$,  the circle with the center at $a$ and radius $1-|a|$ intersects the interval $[0,1]$ at some point $1-\e'$ with $0\le\e'\le\e$ (this holds despite the fact, that the intersection can have two points). This proves an existence of $b\in\bc$ such that $a+b=1-\e'$ and 
$|a|+|b|=1$ with $0\le\e'<1$, so Lemma~\ref{si5} yields $| |a|-a|^2\le\e'\le\e$. 
\end{prf}

Finally, we prove
\begin{dlemma}\label{si8}
Let $w\in\bc^n$ and $v\in\bt^n$. If $0<\frac1c\le|w|$ and $\<v,w\>>0$, then 
\begin{equation}\label{si9}
\|\si(w)-v\|_\infty\le c\sqrt{\|w\|_1(\|w\|_1-\<v,w\>)}.\end{equation}
\end{dlemma}
\begin{prf}
Clearly,
\[
\mbox{\small$\frac1c$}\|\si(w)-v\|_\infty\le\|\ov{w}(\si(w)-v)\|_\infty=\||v\ov{w}|-v\ov{w}\|_\infty.
\]
Since for $u:=v\ov{w}$ we have $\<u,\mathbf{1}\>=\<v,w\>>0$ and $\|u\|_1=\|w\|_1$, the above estimate and Lemma~\ref{si7} yield (\ref{si9}) immediately.
\end{prf}

After proving Proposition~\ref{propalg} we can take on understanding the workings of the algorithm searching for the biunimodular vectors of a given $A\in U(n)$. Recall that we start with an arbitrary vector $V_0\in \bt^n$ and successively apply $P_A$ (given in  Proposition~\ref{propalg}), so that $V_j=P_A^jV_0$, until we reach a $J$ such that 
\begin{equation}\label{delta}
\|AV_J\|_1>n-\delta\ge n-\mbox{\small$\frac1{8}$}
\end{equation}
 with a certain precision $0<\delta\le\frac1{8}$. As we mentioned in Subsection~\ref{sect:algo}, there is no guarantee that for a given starting vector $V_0$ we will find $J\in\bn$ so that  (\ref{delta}) holds. Hence, in practice, the algorithm has to be applied to many random starting vectors in $\bt^n$. 

 Thus, to facilitate the further discussion of the algorithm, we will concentrate on the starting vectors $V_0$ satisfying the initial condition
\begin{equation}\label{v0}
\|AV_0\|_1> n-\mbox{\small$\frac1{8}$}. 
\end{equation}

Under this assumption we can gather our findings in the following

\begin{dproposition}\label{propalg2} Let $A\in U(n)$, $V_0\in\bt^n$ and $V_j=P_A^jV_0$ for $j\in\bn$. If $\|AV_0\|_1> n-\frac1{8}$, then for all $j\in\bn\cup\{0\}$
we have that
\begin{enumerate}
 \item[(a)] $ \|AV_{j}\|_{1}\le\|AV_{j+1}\|_{1}\le n$, 
 \item[(b)]  $V_j\in\bt^n$, $|AV_j|\ge\frac12$ and $|A^*(\si(AV_j))|\ge\frac12$,
\item[(c)] $\|AV_{j+1}\|_1 =\|AV_j\|_1$ iff $V_{j+1}=V_j$,
\item[(d)] $\|V_{j+1}-V_{j}\|_\infty\le 2 \sqrt n (\|AV_{j+1}\|_1 -\|AV_j\|_1)^\frac12$ ,
\item[(e)] $\lim_{j\to\infty}\|V_{j+1}-V_{j}\|_\infty=0$.
\end{enumerate}
\end{dproposition}
\begin{prf}
(a) recalls (\ref{hill}) and $\|A\|_{\infty\to 1}\le n$. (b) is a matter of a simple induction. If $V_j\in\bt^n$ then, by Proposition~\ref{propalg}(b), $|AV_j|\ge\frac12$ and $|A^*(\si(AV_j))|\ge\frac12$, so the signum of the latter, that is $V_{j+1}$, must belong to $\bt^n$. (c) and (d)  follow from
 Proposition~\ref{propalg}. And (e) follows from (d), since (a) assures that the sequence $(\|AV_{j}\|_1)_{j\in\bn}$ is convergent.

\end{prf}

 Since the sequence $(\|AV_{j}\|_1)_{j\in\bn}$ is convergent one can hope, that $V_j$ shall converge to a fix point of $P_A$. This, however, is hard to assure even under the initial condition (\ref{v0}). This lack of proof for convergence can be traced to the square root appearing in Lemma~\ref{si5}, that is hard to remove, since (\ref{si6}) gives $||a|-a|\ge\sqrt{\e |a|}$ for $|a|\le\frac12$. Therefore, we can only fall on $\lim_{j\to\infty}\|V_{j+1}-V_{j}\|_\infty=0$ to draw a standard conclusion
\begin{dcorollary}
If the initial condition (\ref{v0}) holds, then  the cluster points of  $(V_j)_{j\in\bn}$ belong to the set $S_A:=\{v\in \bt^n: P_Av=v\}$.  Moreover, if $S_A$ is finite, then $V_j$ converges to a point in $S_A$. 
\end{dcorollary}
\begin{prf}
Here, we use the notation provided in the two previous propositions and by convergence we mean the convergence in any $\ell^p$-norm, since they are all equivalent. 

 Let $L$ denote the limit of $\|AV_j\|_1$. 
If a subsequence $(V_{j_k})_{k\in\bn}$ of $(V_j)_{j\in\bn}$ converges to $V$, then $P_AV_{j_k}$ converges to $P_AV$, since by (b) of Proposition~\ref{propalg2}, the sequence $(V_j)_{j\in\bn}$ is contained in a closed subset of an open domain, where $P_A$ is continuous. Thus, $\|AV\|_1=\lim_{k\to\infty}\|AV_{j_k}\|_1=L$ and $\|AP_AV\|_1=\lim_{k\to\infty}\|AP_AV_{j_k}\|_1=\lim_{k\to\infty}\|AV_{j_k+1}\|_1=L$, so $\|AV\|_1=\|AP_AV\|_1$. Since $V\in\bt^n$ and   $\|AV\|_1>n-\frac1{8}$, Proposition~\ref{propalg}(c) implies that $P_AV=V$, hence $V\in S_A$.

If $S_A$ is finite, then its points are separated by open balls. Thus, almost all points of  $(V_j)_{j\in\bn}$ are contained in the union of these balls.  By Proposition~\ref{propalg2}(e), there is only one of these balls containing  almost all points of  $(V_j)_{j\in\bn}$, what assures the convergence of $V_j$.

\end{prf}

Since $\bi(A)\subset S_A$, the finiteness of $S_A$ is very hard to establish and, in the Fourier case, brings us back to Conjecture~\ref{bs}.

\section*{Acknowledgements}
Part of this work was carried during several visits of HF to Wroc\l aw, and two visits of ZR to Aachen. We would like to thank our respective hosting organizations. We also thank the referees for providing essential information that made the current understanding of biunimodular vectors more complete.

\vskip1cm

Hartmut F\"uhr \hfill Ziemowit Rzeszotnik

Lehrstuhl A f\"ur Mathematik \hfill Mathematical Institute

RWTH Aachen \hfill  University of Wroc\l aw

52056 Aachen, Germany \hfill 50-384 Wroc\l aw, Poland

\texttt{{fuehr@matha.rwth-aachen.de}} \hfill \texttt{{zioma@math.uni.wroc.pl}}
\end{document}